\documentclass[11pt, a4paper]{article}

\usepackage[T1]{fontenc}
\usepackage[utf8]{inputenc}
\usepackage{xcolor} 
\usepackage{graphicx} 
\usepackage{times} 
\usepackage{amsmath} 
\usepackage{amssymb}  
\usepackage{amsthm}
\usepackage{subcaption}

\usepackage{url}
\usepackage{multirow}
\usepackage{cite}
\usepackage{algorithm}

\usepackage{algorithmic}
\usepackage{microtype} 
\renewcommand{\arraystretch}{1.3}

\definecolor{darkgreen}{rgb}{0.0,0.5,0.0}

\newcommand{\rr}{\mathbb{R}}

\newcommand{\DD}{\mathcal{D}}
\newcommand{\XX}{\mathcal{X}}
\newcommand{\LL}{\mathcal{L}}
\DeclareMathOperator{\tr}{\mathrm{tr}}

\newcommand{\QED}{\hfill\ensuremath{\square}}

\newtheorem{theorem}{Theorem}[section]
\newtheorem{lemma}[theorem]{Lemma}

\newtheorem{proposition}[theorem]{Proposition}

\newtheorem{definition}[theorem]{Definition}

\theoremstyle{remark}
\newtheorem{remark}[theorem]{Remark}
\theoremstyle{plain}

\begin{document}

\title{Learning Parametric Monotone Games}

\author{Alberto Bemporad and Tatiana Tatarenko 
{\renewcommand{\thefootnote}{}\thanks{A. Bemporad is with the IMT School for Advanced Studies, Piazza San Francesco 19, Lucca, Italy. Email: \texttt{\scriptsize alberto.bemporad@imtlucca.it}. T. Tatarenko is with the Department of Control Theory and Intelligent Systems, TU Darmstadt, Germany. E-mails: \texttt{\scriptsize tatiana.tatarenko@tu-darmstadt.de}.   
This work was funded by the European Union (ERC Advanced Research Grant COMPACT, No. 101141351). Views and opinions expressed are however those of the authors only and do not necessarily reflect those of the European Union or the European Research Council. Neither the European Union nor the granting authority can be held responsible for them. The work was also funded by the Deutsche Forschungsgemeinschaft (DFG, German Research Foundation). Project number: 528033031.
}}}

\maketitle
\thispagestyle{empty}

\begin{abstract}
We study the problem of learning from data a parametric Nash equilibrium (NE) problem that is monotone (or strongly monotone) for all parameter values. In the presence of local and shared convex constraints,
monotonicity enables efficient computation of generalized Nash equilibria of the learned game.
We consider two learning scenarios: ($i$) direct learning of the NE problem from samples of the agents' costs, and ($ii$) inverse learning of surrogate agents' costs from samples of their best responses. We propose two methods to solve these tasks. The first is a penalty-based approach that promotes monotonicity of the learned game during training. 
The second, based on a representation theorem we introduce for a broad class of monotone games, parameterizes the agents' costs so that monotonicity of the NE problem is guaranteed by construction, for all parameter values, regardless of the training data used. We illustrate the applicability of the proposed methods on several numerical examples. A Python library and the examples reported in the paper are available at \url{https://github.com/bemporad/learn_monotone_games}.
\end{abstract}

\section{Introduction}

Game theory has become a foundational modeling tool across a wide range of engineering and economic applications where multiple self-interested decision-makers interact with each other~\cite{BYGP22}. Such a non-cooperative multi-agent decision scenario is naturally captured by generalized games, in which several agents choose their decisions by optimizing individual objective functions subject to both local and shared coupling constraints on their decision variables. Such settings arise in economics \cite{Dockner_GameEco_2000}, energy management~\cite{Hall_CDC_2022,LJWA20,WLMCMWW21}, control of multi-agent systems~\cite{Cappello_TCNS_2021}, and in many other domains.
In all of these settings, the cost function and the feasible set of each agent depend on the decisions of the other agents, and the relevant solution concept is the generalized Nash equilibrium (GNE)~\cite{facchinei2010generalized,Bau16}, a state in which no agent has an incentive, or even the possibility, to unilaterally change its decision, given the decisions of the others. 

Despite this ubiquity, computing a GNE is fundamentally harder than solving a single-agent optimization problem. In centralized optimization, {\it convexity} underlies essentially all guarantees of practical interest: existence and uniqueness properties of an optimizer, convergence of different and very efficient solution methods, and tractable optimality certificates via duality~\cite{BV04}. In game-theoretic optimization, the analogous role is played by {\it monotonicity} of the game's pseudo-gradient mapping, or of the variational inequality that characterizes the equilibrium~\cite{facchinei2010generalized}. Monotonicity, and its strengthened variants such as strong or strict monotonicity, are often used as a key condition for establishing existence and uniqueness of equilibria, as well as specific convergence properties of equilibrium-seeking algorithms.
In this sense, monotonicity is to game-theoretic optimization what convexity is to classical optimization: a structural assumption allowing for constructive solution approaches. Building on this vision, a rich body of distributed methods for solving monotone games has been developed, in which agents seek a GNE through local computations and/or limited communication with their neighbors, while retaining the convergence guarantees afforded by monotonicity~\cite{belgioioso2020, TatKam25, TatNed25}.

When dealing with a family of GNE problems (GNEPs), parameterized by a vector of parameters, and when the agents' costs, or their best responses, are only available through black-box queries for given parameter and other agents' decision values, recent work has investigated approaches to learning {\it surrogate models}, in the form of a solution mapping from parameters to equilibrium strategies, or approximate primitives, such as surrogates of the agents' costs or best-response functions.
Such surrogate-based approaches to GNE computation can be divided into two main groups: ($i$) \emph{learning an approximation of the solution mapping} directly, as in~\cite{BT26}, where the Nikaido--Isoda gap function is used as a loss to train a neural network model of a continuous GNE selection, and in~\cite{goktas2023generative}, where Generative Adversarial Equilibrium Solvers (GAES) are used to learn to output GNE of pseudo-games; and ($ii$) \emph{learning a surrogate of the game} (for a given parameter vector) and solving it, as in~\cite{krupa2026learning}, where a GNE is actively learned from pairwise preference data alone, collected
for a specific instance of the game, without the ability to measure either best responses or agents' costs, or in~\cite{han2024noregret}, which instead relies on Gaussian processes as surrogates for unknown utilities, casting equilibrium-seeking as no-regret optimization of an unknown objective. However, even if the underlying problem is monotone, neither approach in the works mentioned above to learn a surrogate of the game guarantees that the learned game is monotone, and the learning phase, being part of the solution procedure, must be repeated for every
different instance of the parameter vector. For learning convex constraint sets from feasible/infeasible samples, we refer the reader to~\cite{SBB25}.

\subsection{Contribution}
In this paper, we present different approaches to learning a parametric NE problem from data that is guaranteed to be monotone (or strongly monotone) for all parameter values. We 
consider both the direct learning from samples of the agents' {\it costs} and inverse learning from samples of their {\it best responses}. These samples can be either measured by querying the agents or generated by evaluating the costs or best-responses of a given parametric NE problem at different parameter values to {\it monotonify} it, by learning a surrogate game that is monotone
for all parameter values. The approach can be considered as the game-theoretic counterpart of linear system identification of possibly nonlinear dynamics for control design~\cite{Lju99}, or constructing convex approximations of nonconvex functions, which is a common approach in optimization and control~\cite{MSA16,DBS05}.

We first propose a penalty-based approach that promotes monotonicity of a generic (input-convex) parameterization of the agents' costs by penalizing violations of the monotonicity condition on the pseudogradient or on its Jacobian on training data. To guarantee monotonicity independently of the training dataset used, a second contribution of this paper
is to analyze a class of convex parametric monotone games whose pseudogradient admits the splitting $F(x,p)=F_s(x,p)+F_a(x,p)$, where the Jacobians $\nabla F_s$, $\nabla F_a$ are symmetric and anti-symmetric, 
respectively, showing that $F_s$ must be the gradient of a convex potential term and $F_a$ affine with skew-symmetric Jacobian. Parameterizing the costs accordingly, in particular by neural network functions of the parameter $p$, we guarantee {\it by construction}, regardless of the training data, that the learned game is monotone for all values of $p$, so it can be solved with standard algorithms for monotone games and generalized NE problems. Although this class does not cover all monotone game, as the skew-symmetric part of the pseudogradient Jacobian is necessarily constant, we show that it covers all monotone convex {\it quadratic} games. For the special case of inverse learning of quadratic monotone games from best-response samples, we show that the learning problem can be solved via convex semidefinite programming.

Unlike our recent work~\cite{BT26}, where we approximate the {\it solution map} of a given parametric GNEP, here we {\it learn the game itself}.
This paper, compared to~\cite{BT26}, has a twofold advantage. First, any convex constraint can be added {\it after} training, while the solution map learnt as in~\cite{BT26} would have to be recomputed if the constraints change. Second, in~\cite{BT26} the learned solution map may violate some constraints for certain values of $p$, due to limitations of the adopted neural network architecture or training algorithm, while here, due to the imposed monotonicity, feasible GNE can be computed efficiently for each given $p$ by leveraging existing algorithms for solving monotone GNE problems. 

The paper is structured as follows. Section~\ref{sec:problem-statement} formulates the direct and inverse learning problems. Section~\ref{sec:monotonicity-penalties} introduces the penalty-based approach to promote monotonicity, and Section~\ref{sec:monotone_parametrization} the parameterization that enforces it by construction, together with the corresponding training problem and its extension to nonsmooth costs and shared constraints. Section~\ref{sec:inverse-quadratic} specializes inverse learning to quadratic monotone games. Section~\ref{sec:numerical-examples} illustrates the proposed approaches on numerical examples of quadratic and nonlinear convex games, as well as the inverse learning of a game-theoretic multivariable controller. 

\subsection{Notation}
We denote by $\rr^n$ the $n$-dimensional Euclidean space.
Given $m$ vectors $v_1,\ldots,v_m$ from $\rr^n$, we denote by $\operatorname{col}(v_1,\ldots,v_m)$ their vertical concatenation. Given a vector $v$, $\|v\|$ denotes its Euclidean norm
and given two vectors $u,v\in\rr^n$, $u\odot v$ denotes their entrywise product.
Given a matrix $A$, $\|A(p)\|=\sqrt{\lambda_{\rm max}(A^\top A)}$ denotes its spectral norm and $\|A\|_F=\sqrt{\sum_{i,j}A_{ij}^2}$ its Frobenius norm.

\section{Problem statement}
\label{sec:problem-statement}
We consider a parametric game setting with $N$ agents, each choosing a decision vector $x_i\in\rr^{n_i}$ to minimize a cost function $\tilde J_i:\rr^n\times\rr^m\to\rr$, $i=1,\ldots,N$, where $\tilde J_i(x,p)$ is the cost of agent $i$ when the agents' decision vector is $x=\operatorname{col}(x_1,\ldots,x_N)\in\rr^n$, $n=\sum_{i=1}^N n_i$, for a given vector of parameters $p\in\rr^m$. In general, $\tilde J_i$ depends on the vector $x_{-i}$ collecting all the components of $x$ but those of $x_i$, leading to the following parametric Nash equilibrium (NE) problem
\begin{equation}
    x_i^\star(p) \in \arg\min_{x_i\in\rr^{n_i}} \left. \tilde J_i(x,p)\right|_{x_{-i}=x_{-i}^\star(p)},\quad i=1,\ldots,N.
\label{eq:NEP}
\end{equation}
In general, the costs $\tilde J_i$ are unknown, e.g., because they encode private preferences or objectives of the agents that cannot be directly measured or evaluated in closed form, and only data in the form of cost or best-response samples are available. To be able to find a solution to the NE problem defined in~\eqref{eq:NEP} in this case, our goal is to learn a {\it surrogate} NE problem from such data, i.e., to learn a parametric model $\{J_i(x,p;\theta)\}_{i=1}^N$ of the agents' costs, where $\theta$ is the vector of parameters of the model, $\theta\in\rr^{n_\theta}$,
and $J_i:\rr^n\times\rr^m\times\rr^{n_\theta}\to\rr$ is convex 
w.r.t. $x_i$ 
for all $x_{-i}$, $p$, and $\theta$. In particular, we are interested in learning the agents' costs, i.e. defining $\theta^*$, so that the resulting convex NE problem, with the costs $\{J_i(x,p) = J_i(x,p;\theta^*)\}_{i=1}^N$, is monotone (or strongly monotone) for all $p$, according to the following definition:
\begin{definition}
Let $F:\rr^n\times\rr^m\to\rr^n$ be the pseudogradient of the game with the costs $\{J_i(x,p)\}_i$,
\[
F(x,p)=\operatorname{col}(\nabla_{x_1}J_1(x,p),\ldots,\nabla_{x_N}J_N(x,p))\in\rr^n.
\]
The game is {\it monotone} with monotonicity constant $\mu\geq0$ (or {\it strongly monotone} if $\mu>0$) if $F$ is a $\mu$-monotone operator, i.e., if
\begin{equation}
 (x-y)^\top(F(x,p)-F(y,p))\geq\mu\|x-y\|^2
\qquad\forall x,y\in\rr^n, \, \forall p\in\rr^m.
\label{eq:monotonicity}
\end{equation}
\label{def:monmotone_game}  
\end{definition}
\noindent In the case that all costs are also twice continuously differentiable w.r.t. $x$ (which implies $F\in C^1$ w.r.t. $x$), monotonicity can be equivalently defined in terms of the Jacobian $\nabla F$ of the pseudogradient (calculated w.r.t. $x$)
as follows:
\begin{equation}
\frac{\nabla F(x,p)+\nabla F(x,p)^\top}{2}\succeq\mu I
\qquad\forall x\in\rr^n, \, \forall p\in\rr^m.
\label{eq:monotone-jacobian}
\end{equation}
i.e., the symmetric part of $\nabla F(x,p)$ is uniformly positive semidefinite (positive definite if $\mu>0$).

\noindent Monotonicity is a key property of NE problems. In particular, it implies that each $J_i$ is convex and $C^1$ w.r.t. $x_i$, and, thus, $x^\star(p)$ is a NE of the game if and only if $F(x^\star(p),p)=0$~\cite[Prop. 1.4.2]{FP03}.
If, in addition, $\mu>0$, then the NE $x^\star(p)$ exists and is unique~\cite[Th. 2.3.3(b)]{FP03}. In the case of a generalized NE problem with shared constraints, strong monotonicity of the pseudogradient is a sufficient condition for the existence and uniqueness of a variational GNE~\cite[Th. 12.1.3]{facchinei2010generalized}.
Moreover, monotonicity is a sufficient condition for the convergence of many algorithms for computing a NE or a variational GNE~\cite[Chapter 12]{FP03}. This motivates us to target monotonicity directly when learning a parametric game or its surrogate from data.

Note that we do not assume that the underlying game $\{\tilde J_i\}_{i=1}^N$ is itself monotone. Instead, we learn the closest monotone surrogate $\{J_i\}_{i=1}^N$ to it. As a result, perfect approximation of the true game cannot be guaranteed in general. Nevertheless, 
when agents' best responses can be queried, the quality of the surrogate can be assessed a posteriori by estimating the deviation between the equilibrium values predicted by the learned surrogate game and each agent's best response to those values.

We will consider two different learning scenarios, described in the next two subsections, depending on whether we have access to samples of the agents' costs or, alternatively, to samples of their best responses.

In the following discussion of the learning problems, we will use $F(x,p;\theta)$ and $\nabla F(x,p;\theta)$, without abuse of notation, to denote the pseudogradient of the game with costs $\{J_i(x,p;\theta)\}_{i=1}^N$ and its Jacobian, respectively.

\subsection{Direct learning from agents' costs}
Assume that we have a set $\DD_J=\{(x_k,p_k,\{\tilde J_i(x_k,p_k)\}_{i=1}^N)\}_{k=1}^K$ of 
data samples of the agents' costs at different decision and parameter vector values $x_k\in\rr^n$, $p_k\in\rr^m$.
The true functions $\tilde J_i$ may be either known or their values may be just observed from data.
The goal is to learn a surrogate NE problem in which the costs $J_i$ approximate $\tilde J_i$ as much as possible,
according to the following loss function 
\begin{equation}
\LL_{DJ}(\theta)=\frac{1}{K}\sum_{k=1}^K \sum_{i=1}^N \left\| J_i(x_k,p_k;\theta) - \tilde J_i(x_k,p_k) \right\|^2
\label{eq:training_problem-J}
\end{equation}
under the condition that the learned game is monotone for all values of $p$.

\subsection{Inverse-learning from best responses}
Suppose now that, instead, we have no information about the costs $\tilde J_i$, but rather have a set 
$\DD_X=\{x_k,p_k,i_k\}_{k=1}^{K}$ of best-response observations, where $i_k\in\{1,\ldots,N\}$ is the index of the agent whose best response is observed at sample $k$ and $x_k = (x_{k,i_k}, x_{k,-i_k})$. In particular,
for each sample $k$, we observe the best response $x_{k,i_k}$ of agent $i_k$ to the other agents' actions $x_{k,-i_k}$ at parameter value $p_k$:
\[
    x_{k,i_k} \triangleq \arg\min_{x_i\in\rr^{n_i}} \left. \tilde J_{i}(x,p_k)\right|_{x_{-i_k}=x_{k,-i_k}}.
\]
This may represent a common practical situation in which we can only observe the agents' actions, under the assumption that they are acting optimally at different parameter values and other agents' decisions, but their costs are not disclosed. The goal is to learn a surrogate NE problem in which the costs $J_i$ provide similar best responses to the true costs $\tilde J_i$ at the observed samples, captured by the following (bilevel) loss
\begin{equation}
    \LL_{DX}(\theta) = \frac{1}{K}\sum_{k=1}^{K}
    \Bigl\| \operatorname*{arg\,min}_{x_{i_k}} J_{i_k}(x\Big|_{x_{-i_k}=x_{k,-i_k}}\!\!\!\!\!\!,p_k, \theta)-{x}_{k,i_k} \Bigr\|^2.
    \label{eq:bilevel}
\end{equation}
We will show in Section~\ref{sec:inverse-quadratic} that, in the special case of inverse learning of quadratic monotone games, the bilevel loss can be avoided and~\eqref{eq:bilevel} can be remapped into either a nonlinear least-squares problem or a convex semidefinite program (SDP).

\subsection{Optional NE samples}
Besides having cost or best response samples, there might be cases in which a second dataset of NE samples is also available, either $\DD_{EJ}=\{(x_k^\star,p_k,\tilde J_i(x^\star_k,p_k))\}_{k=1}^{K^\star}$ or $\DD_{EX}=\{(x_k^\star,p_k)\}_{k=1}^{K^\star}$, where $x_k^\star$ is an unconstrained NE of the game at parameter value $p_k$ and $K^\star$ is the number of NE samples, $K^\star\geq 0$.

Accordingly, when $K^\star>0$, the following losses can be added to the training problems~\eqref{eq:training_problem-J} or~\eqref{eq:bilevel} to promote the matching of the NE samples:
\begin{subequations}
\begin{equation}
\LL_{EF}(\theta) = \frac{\lambda_2}{K^\star} \sum_{k=1}^{K^\star} \|F(x_k^\star,p_k;\theta)\|^2
\label{eq:NE-samples-F}
\end{equation}%
and either
\begin{equation}
\LL_{EJ}(\theta) = \frac{\lambda_1}{K^\star} \sum_{k=1}^{K^\star} \sum_{i=1}^N \left\| J_i(x_k^\star,p_k;\theta) - \tilde J_i(x_k^\star,p_k) \right\|^2
\label{eq:NE-samples-J}
\end{equation}
or
\begin{equation}
\LL_{EX}(\theta) = \frac{\lambda_1}{K^\star}\! \sum_{k=1}^{K^\star} \sum_{i=1}^N \left\| \operatorname*{arg\,min}_{x_i} J_i(x\Big|_{x_{-i}=x^\star_{k,-i}}\!\!\!\!\!\!,p_k, \theta_i)-x_{k,i}^\star  \right\|^2
\label{eq:NE-samples-X}
\end{equation}
\label{eq:NE-samples}%
\end{subequations}
where $\lambda_1,\lambda_2\geq 0$ are penalties that balance the importance of fitting the best-response samples in $\DD_X$ and matching the NE samples in either $\DD_{EJ}$ or $\DD_{EX}$.

Although we mainly focus on learning the costs $J_i$ in a game without constraints (see~\eqref{eq:NEP}), once an unconstrainted monotone game is learned, it can be extended to handle local constraints $x_i\in\XX_i(p)$ and to a GNEP when shared constraints $x\in\XX(p)$ are added to the game.
Note that, in the presence of constraints, the penalty~\eqref{eq:NE-samples-X} is relevant only if the training points $x_k^\star$ are still NEs, not GNEs, otherwise~\eqref{eq:NE-samples-X} must be modified to include the constraint $x\in\XX(p)$ in the inner optimization problem, in the case $\XX(p)$ is already available during the learning phase. Moreover, while~\eqref{eq:NE-samples-J} is relevant no matter the nature of $x_k^\star$ is, when $x_k^\star$ are GNEs the penalty~\eqref{eq:NE-samples-F} should not be included.

So far, we have formulated the learning problems~\eqref{eq:training_problem-J} and~\eqref{eq:bilevel} without explicitly enforcing the monotonicity of the resulting game. 
In the next sections we will introduce two different approaches 
to learning the costs $J_i$ so that they define a parametric monotone game. The first approach is to parameterize each agent's cost function as a generic input-convex parametric model $J_i(x,p;\theta)$ and enforce monotonicity by adding further penalties that promote the monotonicity of the resulting game. The second approach is to enforce monotonicity {\it by construction}, leveraging a suitable parameterization of the models $J_i$, so that the resulting game is monotone for all $p$ no matter what the model parameters $\theta$ and the dataset used are.

\section{Monotonicity enforcement via penalties}
\label{sec:monotonicity-penalties}
Let $\DD_M=\{(x_j,p_j)\}_{j=1}^M$ be a further set of samples of the decision and parameter vectors generated in the range of $x$ and $p$ of interest, such as in the smallest hyperboxes containing the training samples in $\DD_J$ or $\DD_X$, respectively. Note that the samples in $\DD_M$ do not contain any measurement of the agents' costs or best responses.

We consider the following three alternative penalties for promoting monotonicity:
\begin{subequations}
\begin{equation}
\begin{aligned}
\LL_{M1}(\theta)=&\frac{\gamma}{M(M-1)}\sum_{j=1}^M \sum_{h=1, h\neq j}^M \max\left\{0,\mu\|x_h-x_j\|^2\right.\\
&\left.-(x_j-x_h)^\top\left(F(x_j,p_j;\theta)-F(x_h,p_j;\theta)\right)\right\}^2
\end{aligned}
\label{eq:monotonicity-violation-1}
\end{equation}
related to the violation of the monotonicity condition~\eqref{eq:monotonicity},
with $\gamma\gg 1$ a large penalty weight, or
\begin{equation}
\begin{aligned}
\LL_{M2}(\theta)=&\frac{\gamma}{M}\sum_{j=1}^M \max\left\{0,2\mu-\lambda_{\min}\left(\nabla_x F(x_j,p_j;\theta) +\nabla_x F(x_j,p_j;\theta)^\top\right)\right\}^2
\end{aligned}
\label{eq:monotonicity-violation-2}
\end{equation}
related instead to the violation of the monotonicity condition~\eqref{eq:monotone-jacobian}.
The latter is also satisfied whenever the symmetric part of the pseudogradient Jacobian is the Hessian of a convex function with strong-convexity parameter $\mu$ (cf.~Lemma~\ref{lem:splitting} below), which leads to the following third alternative penalty to promote monotonicity: introduce an auxiliary, twice differentiable, input-convex network $\Phi(x,p;\theta)$ and penalize
\begin{equation}
\begin{aligned}
    \LL_{M3}(\theta) =&\frac{\gamma}{M}\sum_{j=1}^M \left\| \nabla_x F(x_j,p_j;\theta)+
\nabla_x F(x_j,p_j;\theta)^\top 
- \nabla^2_x \Phi(x_j,p_j;\theta) -2\mu I\right\|_F^2.
\end{aligned}
\label{eq:monotonicity-violation-3}
\end{equation}%
\end{subequations}
Note that in~\eqref{eq:monotonicity-violation-3}, $\theta$ also includes the parameters of $\Phi$, which is only auxiliary and can be discarded after training. 

Finally, the learning problem can be formulated as follows:
\begin{equation}
\min_\theta \LL_{D}(\theta) + \LL_M(\theta) + \LL_{E}(\theta) + \LL_{EF}(\theta) + \frac{\rho}{2}\|\theta\|^2
\label{eq:learning-problem-point}
\end{equation}
where $\LL_D$ is either $\LL_{DJ}$ or $\LL_{DX}$, $\LL_M$ is one of the three penalties $\LL_{M1}$, $\LL_{M2}$, or $\LL_{M3}$, and $\LL_E$ is either $\LL_{EJ}$ or $\LL_{EX}$ if NE samples are available, and zero otherwise,
and similarly $\LL_{EF}(\theta)$ is zero if no NE samples are given. The last term is a standard $\ell_2$-regularization introduced to prevent overfitting.

\section{Monotonicity enforcement by construction}
\label{sec:monotone_parametrization}
The main drawback of the monotonicity-enforcing approach described in Section~\ref{sec:monotonicity-penalties} is that
no guarantee exists that the resulting game is monotone for all values of the parameter $p$, as the result largely
depends on the weight $\gamma$ in $\LL_M$ and on the richness of the dataset $\DD_M$. 
In order to construct agents' costs $J_i$ that enforce the monotonicity condition~\eqref{eq:monotone-jacobian} by construction, in this section we propose a different approach. 

The following lemma (Lemma~\ref{lem:splitting}) shows that a sufficient condition for monotonicity is that the pseudogradient $F$ can be split into a symmetric and a skew-symmetric part.
For simplicity of notation, in the following we omit the dependence of $F$ on the parameter $p$ and on the model parameters $\theta$.

\begin{lemma}
\label{lem:splitting}
Let $F:\rr^n\to\rr^n$ be the pseudogradient of a game with
costs $J_i:\rr^n\to\rr$, $J_i\in C^2$ and convex in $x_i$, $\forall i=1,\ldots,N$,
and assume that
\begin{equation}
F(x)=F_s(x)+F_a(x)
\label{eq:F(x)-split}
\end{equation}
where $F_s,F_a:\rr^n\to\rr^n$, $F_s\in C^1$, $F_a\in C^2$ are such that
$\nabla F_s(x)$ is symmetric and positive semidefinite for all $x$, and
$\nabla F_a(x)$ is skew-symmetric for all $x$. Then:
\begin{enumerate}
\item[($i$)] $F_s(x)=\nabla\Psi(x)$ for some convex function
$\Psi:\rr^n\to\rr$, $\Psi\in C^2$;
\item[($ii$)] $F_a(x)=Sx+r$ is affine, with $S^\top=-S$,
$r\in\rr^n$;
\item[($iii$)] the game is monotone and each cost function necessarily has
the form
\begin{equation}
J_i(x)=\Psi(x)+\sum_{j\neq i}x_i^\top S_{ij}x_j+r_i^\top x_i
+\phi_i(x_{-i})
\label{eq:J_i-form}
\end{equation}
where $\phi_i:\rr^{n-n_i}\to\rr$ and
$S=[S_{ij}]$, $r=\operatorname{col}(r_1,\ldots,r_N)$ are partitioned
according to the dimensions $n_1,\ldots,n_N$ of the agents' decision vectors.
\end{enumerate}
\end{lemma}

\begin{proof}
($i$) Since $\nabla F_s(x)$ is symmetric for all $x$ and $\rr^n$ is
simply connected, by the Poincar\'e lemma $F_s$ is a gradient field:
there exists $\Psi\in C^2(\rr^n)$ with $F_s=\nabla\Psi$. Since
$\nabla^2\Psi(x)=\nabla F_s(x)\succeq0$ for all $x$, $\Psi$ is convex.

($ii$) The condition $\nabla F_a(x)+\nabla F_a(x)^\top=0$ for all $x$ is equivalent to
\begin{equation}
\frac{\partial (F_a)_i}{\partial x_j}(x)
+\frac{\partial (F_a)_j}{\partial x_i}(x)=0
\qquad i,j=1,\ldots,n.
\label{eq:Killing}
\end{equation}
Differentiating with respect to $x_k$,
exploiting the symmetry of Hessian matrices,
and cyclically permuting the indices
$(i,j,k)$ yields
\[
\begin{aligned}
\partial_k\partial_j (F_a)_i 
&=-\partial_k\partial_i (F_a)_j 
&=-\partial_i\partial_k (F_a)_j\\
&=\partial_i\partial_j (F_a)_k 
&=\partial_j\partial_i (F_a)_k \\
&=-\partial_j\partial_k (F_a)_i & = -\partial_k\partial_j (F_a)_i
\end{aligned}
\]
which in turn implies that $\partial_k\partial_j (F_a)_i=0$ for all $j,k,i=1,\ldots,n$, i.e., all second derivatives of $(F_a)_i$ are zero, for all $i=1,\ldots,n$.
Therefore, $F_a(x)=Sx+r$ with $S=\nabla F_a$ constant and $S^\top=-S$ due to the assumed skew-symmetry of $F_a$.

($iii$) By ($i$)--($ii$), the pseudogradient has the form
\[
F(x)=\nabla\Psi(x)+Sx+r
\qquad
\nabla F(x)=\nabla^2\Psi(x)+S
\]
whose symmetric part is $\tfrac{1}{2}(\nabla F(x)+\nabla F(x)^\top)=\nabla^2\Psi(x)\succeq0$; hence $F$ is monotone.
Moreover, since $F_i(x)=\nabla_{x_i}J_i(x)$, integrating with respect to
$x_i$ for fixed $x_{-i}$ gives
\[
J_i(x)
=\Psi(x)+x_i^\top\sum_{j\neq i}S_{ij}x_j
+\tfrac12x_i^\top S_{ii}x_i+r_i^\top x_i+\phi_i(x_{-i})
\]
for some function $\phi_i$ such that $\phi_i(x_{-i})$ is independent of $x_i$, and $x_i^\top S_{ii}x_i=0$ because $S_{ii}$ is skew-symmetric, leading to the form~\eqref{eq:J_i-form}.
\end{proof}
The convex potential term $\Psi(x)$ can be a rather arbitrary function in $C^2$. Special cases of $\Psi(x)$ include general input-convex neural networks,
and the smooth approximation of the indicator function of the set defined by known constraints on agents' actions; for example, for a box constraint $x\in[x_{\min},x_{\max}]$, the smooth barrier function
\[
\Psi(x)
=\beta\varepsilon
\left(
\log\!\left(1+e^{(x_{\min}-x)/\varepsilon}\right)
+
\log\!\left(1+e^{(x-x_{\max})/\varepsilon}\right)
\right)
\]
where $\beta\gg 1$ and $\varepsilon\ll 1$. 

\begin{remark}
\label{rmk:nonsmooth-Psi}
A simple extension to allow $J_i\not\in C^2$ is to consider an arbitrary convex potential 
term $\Phi:\rr^n\to\rr$ and set $J_i(x)=\Phi(x)+H_i(x)$, with $H_i\in C^2$ and the game defined by the costs $H_i$ satisfying the splitting assumption~\eqref{eq:F(x)-split}, possibly with a zero symmetric part of the associated pseudogradient if no further potential terms are required. By Lemma~\ref{lem:splitting}, the resulting game is monotone and has the form~\eqref{eq:J_i-form}, with $\Psi(x)=\Phi(x)+\Psi_H(x)$, where $\Psi_H\in C^2$ is the convex potential due to the smooth part $H_i$. An example of extending the construction to nonsmooth convex cost functions is to select 
\[
    \Phi(x) = \max_{j=1,\ldots,M} \Phi_j(x)
\]
where each $\Phi_j(x)$ is a smooth convex function. By introducing a slack variable $y_i$ for each agent, the associated NE problem can be rewritten as the following jointly convex merely monotone generalized Nash equilibrium (GNE) problem with shared constraints:
\[
    \begin{aligned}
    \min_{x_i,y_i}\ & y_i+ \Psi(x)+\sum_{j\neq i}x_i^\top S_{ij}x_j+r_i^\top x_i
+\phi_i(x_{-i})
    \\
    \text{s.t. } &\Phi_j(x)\leq y_i,\quad j=1,\ldots,M,\quad i=1,\ldots,N.
    \end{aligned}
\]
At the equilibrium, each agent $i$ has $y_i^\star = \max_{j=1,\ldots,M} \Phi_j(x^\star)$ due to minimizing $y_i$.

Special cases of this formulation include convex piecewise quadratic terms $\Phi(x)$, which can be expressed as the maximum of a finite number of quadratic functions $\Phi_j(x)$, and convex piecewise linear terms $\Phi(x)$, which can be expressed as the maximum of a finite number of affine functions $\Phi_j(x)$. 
\end{remark}

\begin{remark}
\label{rmk:splitting_not_necessary}
The splitting assumption~\eqref{eq:F(x)-split} is sufficient for monotonicity but not
necessary. Consider the two-player game with scalar decisions and costs
\[
J_1(x)=\tfrac12x_1^2+x_1\sin x_2
\qquad
J_2(x)=\tfrac12x_2^2-x_2\sin x_1
\]
each smooth and strongly convex in the agent's own variable. Its
pseudogradient and Jacobian are
\[
F(x)=\begin{bmatrix}x_1+\sin x_2\\ x_2-\sin x_1\end{bmatrix}
\qquad
\nabla F(x)=\begin{bmatrix}1 & \cos x_2\\ -\cos x_1 & 1\end{bmatrix}.
\]
The eigenvalues of the symmetric part of $\nabla F(x)$ are given by
\[
    \det\begin{bmatrix}
         \lambda - 1 & \frac12(\cos x_2-\cos x_1) \\ \frac12(\cos x_1-\cos x_2) & \lambda - 1 
    \end{bmatrix} = 0
\]
or
$
    \lambda^2 - 2\lambda + 1 - \frac14(\cos x_2-\cos x_1)^2 = 0,
$
which gives $\lambda = 1 \pm \frac12|\cos x_2-\cos x_1|$. As both
eigenvalues are nonnegative, the game is monotone.
However, the skew-symmetric part
\[
\frac12\bigl(\nabla F(x)-\nabla F(x)^\top\bigr)
=\frac{\cos x_1+\cos x_2}{2}
\begin{bmatrix}0&1\\-1&0\end{bmatrix}
\]
depends on $x$; hence, by Lemma~\ref{lem:splitting}, no decomposition $F=F_s+F_a$ with
$\nabla F_s$ symmetric and $\nabla F_a$ skew-symmetric everywhere exists
for this game, otherwise the skew-symmetric part would be constant.
\QED
\end{remark}
While Remark~\ref{rmk:splitting_not_necessary} shows that the splitting assumption~\eqref{eq:F(x)-split} is not necessary in general for monotonicity, the following lemma
shows that it is indeed necessary when restricting to the class of monotone {\it quadratic} games. 

\begin{lemma}
\label{lem:quadratic_complete}
If the costs $J_i$ are convex quadratic functions of $x$ and the game is monotone with
monotonicity constant $\mu\geq0$, then $F(x)=F_s(x)+F_a(x)$ with
\[
F_s(x)=(C^\top C+\mu I)x
\qquad
F_a(x)=(D-D^\top)x+q
\]
for some upper-triangular matrix $C\in\rr^{n\times n}$ and block strictly upper-triangular matrix
$D\in\rr^{n\times n}$, where blocks refer to the partition of $x$ into the agents' decision
vectors $x_1,\ldots,x_N$, i.e., $D=[D_{ij}]$ with
$D_{ij}\in\rr^{n_i\times n_j}$ and $D_{ij}=0$ for $i\geq j$. Moreover, by letting the resulting pseudogradient matrix be
\begin{subequations}
\begin{equation}
A=C^\top C+D-D^\top+\mu I
\label{eq:pseudogradient-matrix}
\end{equation}
the individual cost functions can be written, without loss of generality, as
\begin{equation}
    J_i(x) = \frac12 x_i^\top A_{ii} x_i + \sum_{j\neq i} x_i^\top A_{ij} x_j + q_i^\top x_i + \frac12 x_{-i}^\top A_{-i,-i}x_{-i}.
\label{eq:Ji-quadratic}
\end{equation}
\label{eq:quadratic_model}%
\end{subequations}
\end{lemma}

\begin{proof}
Since the game is quadratic, its pseudogradient $F(x)=Ax+q$. By splitting $A=\frac{A+A^\top}{2}+\frac{A-A^\top}{2}$ we get $F(x)=F_s(x)+F_a(x)$, with $F_s(x)=\frac{A+A^\top}{2}x$ and $F_a(x)=\frac{A-A^\top}{2}x+q$. 
By assumption, the matrix $M=\frac{A+A^\top}{2}-\mu I$ is symmetric and positive semidefinite,
so $M=B^\top B$ for some matrix $B$. By taking the QR factorization $B=QC$ (with $Q^\top Q=I$
and $C$ upper-triangular) we get
$M=C^\top C$ and hence $F_s(x)=(C^\top C+\mu I)x$. Moreover,
each diagonal block $A_{ii}=\nabla^2_{x_ix_i}J_i$ is a Hessian and hence
symmetric, so the diagonal blocks $\frac{A_{ii}-A_{ii}^\top}{2}$ of the
skew-symmetric matrix $\frac{A-A^\top}{2}$ vanish. Next, let $D$ be its block
strictly upper-triangular part, i.e., $D_{ij}=\frac{A_{ij}-A_{ji}^\top}{2}$
for $i<j$ and $D_{ij}=0$ for $i\geq j$. Then
$\frac{A-A^\top}{2}=D-D^\top$ and hence $F_a(x)=(D-D^\top)x+q$. Finally, since the $i$-th block of $F$ is
$\nabla_{x_i}J_i(x)=A_{ii}x_i+\sum_{j\neq i}A_{ij}x_j+q_i$, integrating with
respect to $x_i$ determines $J_i$ up to an additive term depending only on
$x_{-i}$. As such a term does not affect the optimization of player $i$, it
can be normalized to $\frac12 x_{-i}^\top A_{-i,-i}x_{-i}$,
giving~\eqref{eq:Ji-quadratic}.
\end{proof}

\subsection{Training problem}
Building on the parameterization of Lemma~\ref{lem:quadratic_complete}, we now let $C$, $D$, and $q$ depend on the parameter vector $p$ and train them from data so as to fit a parametric family of monotone quadratic games.
Similarly to~\cite{SBB25}, we parameterize each nonzero entry of $C$, $D$, and $q$ as a function of the parameter vector $p$, such as a neural network, and we also parameterize the potential function $\Psi(x,p;\theta)$ as an input-convex neural network, as well as neural-network functions $\phi_i(x_{-i},p;\theta)$. 
In this case, $\theta$ collects all the parameters of these neural networks. Other classes of parametric functions can be used as well.
The resulting learning problem, which reads as in~\eqref{eq:learning-problem-point}
with $\LL_M(\theta)=0$, provides a model $J(x,p;\theta)$ that trades off between fitting the data and ensuring monotonicity by construction for all values of $p$.

\begin{remark}
The monotonicity parameter $\mu$ is a degree of freedom of the learning approach: it trades off the flexibility of the model (small $\mu$, even $\mu=0$) against the numerical and convergence properties that the online solution method should enjoy for any instance of $p$ (large $\mu$), especially when shared constraints are added to the game after the agents' costs have been learned. Note that we choose here a uniform $\mu$ for all values of $p$, as parameter-dependent monotonicity can already be enforced by properly selecting $\Psi(x,p;\theta)$, for example by setting $\Psi(x,p;\theta)=\frac{\tilde\mu(p;\theta)}{2}\|x\|^2+\tilde\Psi(x,p;\theta)$ with $\tilde\Psi$ convex in $x$ and $\tilde\mu(p;\theta)\geq 0$ for all $p$.
\end{remark}

\section{Inverse-learning of monotone quadratic games}
\label{sec:inverse-quadratic}

We now specialize the inverse-learning problem of Section~\ref{sec:problem-statement} to convex quadratic games, showing that the bilevel loss~\eqref{eq:bilevel} can be avoided altogether and the training problem recast as a tractable nonlinear least-squares problem or, in the monotone case, as a convex semidefinite program.

\subsection{Convex quadratic games}
Consider again the learning problem in~\eqref{eq:bilevel} with convex quadratic costs $J_i$ as in~\eqref{eq:Ji-quadratic}. The low-level problem is solved by $ \bar x_{k,i}$ such that
\begin{equation}
    \nabla_{x_i} J_i = A_{ii} \bar x_{k,i} + A_{i,-i} x_{k,-i} + q_i = 0
\label{eq:zero-gradient}
\end{equation}
which gives the affine best-response expression $\bar x_{k,i} = -A_{ii}^{-1}\bigl( A_{i,-i}\, x_{k,-i} + q_i \bigr)$.
We provide below three different approaches to address the learning problem~\eqref{eq:bilevel} based on such an affine relation between $\bar x_{k,i}$ and $x_{k,-i}$.

\subsubsection{Nonlinear least-squares formulation}
Given the pseudogradient-matrix parameterization in~\eqref{eq:pseudogradient-matrix}, the learning problem can be formulated as the following nonlinear least-squares problem:
\begin{equation}
    \begin{aligned}
    \min_\theta &\frac{\rho}{2}\|\theta\|^2 +
        \frac{1}{K}\sum_{k=1}^{K}
        \bigl\| A_{ii}^{-1}\bigl( A_{i,-i} x_{k,-i} + q_i \bigr) + {x}_{k,i} \bigr\|_2^2
    \end{aligned}
\label{eq:nls-monotone}
\end{equation}
where $i=i_k$ and we have omitted the explicit dependence of $A_{ii}$, $A_{i,-i}$, and $q_i$ on $p_k$ and $\theta$ for brevity, and $A$ is parameterized as in~\eqref{eq:pseudogradient-matrix} to ensure the monotonicity of the learned game.

\subsubsection{Semidefinite programming formulation}
By weighting each residual in~\eqref{eq:nls-monotone} by $A_{ii}$, which
corresponds to considering the violation of the zero-gradient condition~\eqref{eq:zero-gradient}
instead of the best-response errors as in~\eqref{eq:nls-monotone}, and directly parameterizing the pseudogradient matrix $A$ as a free parameter independent of $p$ and $q=q_0+q_1 p$ as affine, $\theta=(A,q_0, q_1)$, problem~\eqref{eq:nls-monotone} can be reformulated as the following convex SDP problem in $\theta$:
\begin{equation}
    \begin{aligned}
    \min_\theta &\frac{\rho}{2}\|\theta\|^2 \!+\!
        \frac{1}{K}\sum_{k=1}^{K}
        \bigl\| A_{i,-i} x_{k,-i} + q_{0i}+q_{1i} p_k + A_{ii}{x}_{k,i} \bigr\|_2^2\\
    \textrm{s.t. } &\tfrac{1}{2}\bigl( A + A^\top \bigr) \succeq \mu I\\
    & A_{ii} = A_{ii}^\top,\  i = 1, \ldots, N\\
    & \tr(A)=n
    \end{aligned}
\label{eq:sdp-monotone}
\end{equation}
which can be solved to global optimality. Note that the constraint $\tr(A)=n$ is added to avoid the trivial solution $\theta=(0,0,0)$ when $\mu=0$, and only corresponds to a scaling of the resulting
learned cost functions.

\subsubsection{Least-squares/semidefinite programming formulation}
Problem~\eqref{eq:sdp-monotone} can be further manipulated into the cascade of an unconstrained least-squares problem and an SDP problem. To this end, let us introduce the following new variables:
\begin{equation}
    \begin{aligned}
    P_i \triangleq &-A_{ii}^{-1} A_{i,-i} \in \rr^{n_i \times n_{-i}}\\
    f_{ji} \triangleq &-A_{ii}^{-1} q_{0i} \in \rr^{n_i}, \ j=0,1.
    \end{aligned}
\end{equation}
Substituting the affine best response into \eqref{eq:bilevel}, the bilevel problem collapses to the \emph{linear} least-squares problem
\begin{equation}
    \min_{\{P_i,\, f_i\}_{i=1}^{N}}
    \frac{1}{K}\sum_{i=1}^{N}\sum_{k:\ i_k=i}
    \bigl\| P_i x_{k,-i} + f_{i0}+f_{i1}p_k - {x}_{k,i} \bigr\|_2^2
    \label{eq:ls}
\end{equation}
which decouples across agents: each $(P_i, f_i)$ is obtained from an ordinary least-squares fit of agent $i$'s measured best responses on the opponents' decisions. After solving~\eqref{eq:ls}, the original pseudogradient matrix $A$ and vectors $q_0,q_1$ are related to the optimal least-squares estimates $(P_i, f_{i0}, f_{i1})$ by the linear relations
\begin{equation}
    A_{ii}\, P_i = -\,A_{i,-i},
    \qquad
    A_{ii}\, f_{ji} = -\,q_{ji}, \quad j=0,1.
    \label{eq:A-reconstruction}
\end{equation}
Since no convexity or monotonicity constraints are imposed in~\eqref{eq:ls}, to recover a monotone game we can solve the following SDP problem:
\begin{equation}
    \begin{aligned}
    \min_{A, q} &
    \sum_{i=1}^{N}
    \Bigl\| A_{ii} \bigl[ P_i\ f_{i0} \ f_{i1} \bigr] + \bigl[ A_{i,-i}\ q_{i0}\ q_{i1} \bigr]  \Bigr\|_F^2\\
    \text{s.t. } & \textrm{constraints in~\eqref{eq:sdp-monotone}}
    \label{eq:two-stage}
    \end{aligned}
\end{equation}
whose cost function is the sum of the squared residuals of the linear equations~\eqref{eq:A-reconstruction}. 
Problem \eqref{eq:two-stage} is again an LMI-constrained least-squares problem, but unlike
\eqref{eq:sdp-monotone}, its cost function is independent of the number of samples $K$.
The advantage of such a two-stage approach is that the first stage~\eqref{eq:ls} is a standard unconstrained least-squares problem that can be solved very efficiently, and the second stage~\eqref{eq:two-stage} is an semidefinite program whose size depends only on the number of agents and their decision dimensions, but not on the number of samples $K$.

\subsubsection{Extensions}
The formulation~\eqref{eq:nls-monotone} can be extended to appropriate non-quadratic monotone games by modeling the agents' costs as in Section~\ref{sec:monotone_parametrization}, so that the pseudogradient of the game is $F(x,p;\theta)=A(p;\theta)x+q(p;\theta)+\nabla_x\Psi(x,p;\theta)$, with $A(p;\theta)$ parameterized as in~\eqref{eq:pseudogradient-matrix} and $\Psi$ convex in $x$, and minimizing the residual of the first-order optimality condition $\nabla_{x_i} J_i = 0$ in a nonlinear least-squares sense, similarly to the loss adopted in~\eqref{eq:sdp-monotone}:
\begin{equation}
    \begin{aligned}
    \min_\theta\ &\frac{\rho}{2}\|\theta\|^2 + \frac{1}{K}\sum_{k=1}^{K} \Big\|
    \nabla_{x_i} \Psi(x_k,p_k;\theta)
    + A_{ii}(p_k;\theta) x_{k,i} + A_{i,-i}(p_k;\theta) x_{k,-i} + q_i(p_k;\theta)  \Big\|_2^2\\ &  +\frac{1}{K}\sum_{k=1}^{K}\left(\tr(A(p_k;\theta))-n\right)^2
    \end{aligned}
\label{eq:nls-nonquadratic}
\end{equation}
where $i=i_k$  and the last term 
plays the role of the constraint $\tr(A)=n$ in~\eqref{eq:sdp-monotone} to avoid the trivial zero solution. The resulting problem is a nonlinear least-squares problem in $\theta$.

We can also extend the above approach to  monotone {\it generalized} NE problems, where the agents' decisions are subject to a shared polyhedral constraint set $\mathcal{X}$ 
\begin{equation}
    \mathcal{X} =\bigl\{x \in \rr^{n}:\ A_x x \leq b_x\bigr\},
    \qquad A_x \in \rr^{m \times n},\ b_x \in \rr^{m}
\end{equation}
where $A_x$ and $b_x$ are known or learnable. In this case, the best response of agent $i=i_k$ to $\bar{x}_{k,-i}$ is the solution of the quadratic program
\begin{equation}
    \begin{aligned}
    \bar x_i(\theta, p_k, x_{k,-i};\theta) \in&
    \operatorname*{arg\,min}_{x_i} J_i(x,p_k;\theta)\\
    \textrm{s.t. }& A_x x \leq b_x,\ \mbox{where}\ x_{-i} = {x}_{k,-i}
    \end{aligned}
    \label{eq:inner-qp}
\end{equation}
and the inverse problem becomes the following bilevel program
\begin{equation}
    \min_{\theta} \frac{1}{K}\sum_{k=1}^{K}\Big\|
    \operatorname*{arg\,min}_{x_i:\; A_x x \leq b_x,\; x_{-i}={x}_{k,-i}} J_i(x,p_k;\theta)
    -{x}_{k,i}\Big\|_2^2
    \label{eq:bilevel-constrained}
\end{equation}
which can be solved as a nonlinear least squares problem by employing differentiable quadratic programming (QP) solvers~\cite{TM24,AABBDK19} when the costs $J_i$ are parameterized as in~\eqref{eq:quadratic_model}.

\section{Lipschitz constant of the pseudogradient}
\label{sec:lipschitz}
In the previous sections, we have shown how to fit a parametric surrogate monotone game to data, with the goal of learning a model that can predict the agents' behavior for new values of the parameter $p$ and other agent's decisions $x_{-i}$. When the goal is to compute a Nash equilibrium of the surrogate game for a given $p$, and, in particular, a variational GNE (v-GNE) when imposing shared constraints, different online algorithms can be used. 
Solution methods include minimizing the residual of the joint KKT conditions~\cite[Sect. 3.1]{Bem25}, or using iterative methods tailored to monotone games, such as Korpelevich's extragradient method~\cite{Kor76} or gradient play with its accelerated versions~\cite{NS11,KLL22} in the case $\mu>0$. The convergence of iterative algorithms often depends on the inverse of the Lipschitz constant $L(p)$ of the pseudogradient $F(x,p;\theta^\star)=A(p;\theta^\star)x+q(p;\theta^\star)+\nabla_x\Psi(x,p;\theta^\star)$, where $\theta^\star$ is the learned parameter vector, $A(p;\theta^\star)$ is parameterized as in~\eqref{eq:pseudogradient-matrix} and $\Psi$ is convex in $x$. 

In the special case of a quadratic game ($\Psi(x,p;\theta^\star)\equiv 0$), the Lipschitz constant of the pseudogradient is simply given by the spectral norm of the pseudogradient matrix $A$, i.e., $L(p)=\|A(p;\theta^\star)\|$. In the next proposition, we show how to compute an upper bound on $L(p)$ for a specific class of input-convex neural networks $\Psi(x,p;\theta)$. 
For further results on estimating Lipschitz constants of the gradient of neural networks, see, e.g.,~\cite[Theorem 4]{SF20} and~\cite{ESF24}. We will drop the explicit dependence on $p$ and $\theta^\star$ for notation simplicity.

\begin{proposition}
\label{prop:ICNN-lipschitz}
Let $\Psi(x)$ be an input-convex neural network with $Y$ layers, softplus activation $\sigma(t)=\log(1+e^t)$, and linear bypass:
\[
\begin{aligned}
z_1&=\sigma(W_0x+b_0),\quad
z_{k+1}=\sigma(W_kz_k+U_kx+b_k),\quad k=1,\ldots,Y-1
\\
\Psi(x)&=w_Y^\top z_Y+U_Yx+c
\end{aligned}
\]
where $W_k\geq0$ entrywise for $k=1,\ldots,Y-1$ and $w_Y\geq0$ enforce convexity. Define the following scalar quantities recursively:
\[
\begin{aligned}
G_1&=\|W_0\|,\quad &G_{k+1}&=\|W_k\|G_k+\|U_k\|\\
H_1&=\tfrac14G_1^2,\quad &H_{k+1}&=\tfrac14G_{k+1}^2+\|W_k\|H_k,\quad k=1,\ldots,Y-1.
\end{aligned}
\]
Then $\nabla\Psi$ is Lipschitz continuous with constant $L_\Psi\leq\|w_Y\|_2H_Y$.
\end{proposition}
\begin{proof}
Let $p_k=W_{k-1}z_{k-1}+U_{k-1}x+b_{k-1}$, $z_k=\sigma(p_k)$, and consider its Jacobian $Dp_k=W_{k-1}Dz_{k-1}+U_{k-1}$ w.r.t.~$x$. We want to show that $\|Dp_k(x)\|\leq G_k$ for all $x$ by induction on $k$. 
Since $p_1=W_0x+b_0$ is affine in $x$, the base case $\|Dp_1\|=\|W_0\|=G_1\leq G_1$ trivially holds.
Assume by induction that $\|Dp_{k-1}(x)\|\leq G_{k-1}$. Since $z_{k-1}=\sigma(p_{k-1})$ is applied entrywise, $Dz_{k-1}=\operatorname{diag}\bigl(\sigma'(p_{k-1})\bigr)Dp_{k-1}$ and, as $\sigma'(t)\in(0,1)$ for every $t$, $\|Dz_{k-1}(x)\|\leq\|Dp_{k-1}(x)\|\leq G_{k-1}$ for all $x$. The triangle inequality and submultiplicativity of the spectral norm then give
\[
\begin{aligned}
\|Dp_k(x)\|&\leq\|W_{k-1}Dz_{k-1}(x)\|+\|U_{k-1}\|\leq\|W_{k-1}\|\,\|Dz_{k-1}(x)\|+\|U_{k-1}\|\\
&\leq\|W_{k-1}\|G_{k-1}+\|U_{k-1}\|=G_k
\end{aligned}
\]
for all $x$, which closes the induction. 

Let us now compute a bound on the second derivative of one ICNN layer. Consider two arbitrary unit vectors $u,v$ and let $D^2z_k(x)[u,v]\in\rr^{\dim z_k}$ denote the vector with entries $u^\top\nabla^2(z_{k}(x)_i)v$. It is easy to show that
\[
    D^2z_k(x)[u,v]=\sigma''(p_k)\odot(Dp_ku)\odot(Dp_kv)+\sigma'(p_k)\odot\bigl(W_{k-1}D^2z_{k-1}(x)[u,v]\bigr).
\]
Bounding $\sigma''\leq\tfrac14$, $\sigma'\leq1$, and using $\|a\odot b\|_2\leq\|a\|_2\|b\|_2$ and submultiplicativity of the spectral norm,
\[
\begin{aligned}
\|D^2z_k(x)[u,v]\|_2&\leq\tfrac14\|Dp_ku\|_2\|Dp_kv\|_2+\|W_{k-1}\|\,\|D^2z_{k-1}(x)[u,v]\|_2\\
&\leq\tfrac14G_k^2+\|W_{k-1}\|H_{k-1}=H_k
\end{aligned}
\]
which is the claimed induction ($H_1$ is the base case, $z_1$ having no incoming $W$-term). Since $\Psi(x)=w_Y^\top z_Y(x)+U_Yx+c$ with the bypass $U_Yx$ affine and hence contributing no curvature,
\[
|u^\top\nabla^2\Psi(x)v|=|w_Y^\top D^2z_Y(x)[u,v]|\leq\|w_Y\|_2\|D^2z_Y(x)[u,v]\|_2\leq\|w_Y\|_2H_Y
\]
for all unit vectors $u,v$ and all $x$, so $\|\nabla^2\Psi(x)\|\leq\|w_Y\|_2H_Y$.
\end{proof}

\noindent Finally, a bound on the Lipschitz constant $L$ of the pseudogradient $F$ for a given surrogate game instance follows from the triangle inequality:
\begin{equation}
    L\leq\|w_Y\|_2H_Y+\|A\|.
\label{eq:L(p)}
\end{equation}

\section{Numerical examples}
\label{sec:numerical-examples}
In all the experiments reported below, the exact equilibrium $x^\star(p)$ of the true game at each test value of $p$ is computed by the KKT-residual method implemented in the NashOpt package~\cite{Bem25} that,
for unconstrained games, reduces to finding a zero of the pseudogradient.

All numerical tests were run in Python 3.11
on a MacBook Pro with Apple M5 Max (18 CPU cores). Models are trained using the \texttt{jax-sysid} package~\cite{Bem25b}, by running 1000 Adam iterations followed by at most 5000 L-BFGS iterations from 18 random initializations in parallel, keeping the model with the smallest loss on a separate validation dataset of $K_{\rm val}$ samples.
Input-convexity of neural networks is enforced by applying the softplus function both to the weights of the layers that are required to be nonnegative and to the hidden layers. Unless otherwise specified, $\gamma=10$ and $\rho=10^{-8}$. 

To verify the quality of approximate solutions $\hat x(p)$ for a given $p$, we
evaluate each best response (BR) by fixing $x_{-i}=\hat x_{-i}(p)$ and optimizing
the underlying agents costs $\tilde J_i$, possibly subject to shared constraints (unless specified otherwise), NE of the underlying true game $\{\tilde J_i\}_{i=1}^N$ are evaluated via the minimization of the squared pseudogradient norm $\|\nabla \tilde F(x,p)\|^2$ or, in the case of shared constraints, a variational GNE is evaluated by minimizing the KKT residuals as described in~\cite{Bem25}. 
Both the BR and NE errors reported in the tables are evaluated as $\frac{1}{NK_{\rm test}}\sum_{k=1}^{K_{\rm test}}\sum_{i=1}^{N}\|\hat x_i(p_k) - \bar x_i(p_k)\|_2$ over the test samples $p_k$, where $\bar x_i(p_k)$ are either the ground truth BR to $\hat x_{-i}(p)$ or NE, respectively. 

The code to reproduce the examples reported in this section is available at \url{https://github.com/bemporad/learn_monotone_games}.

\subsection{Convex quadratic games}
We consider a randomly generated parametric quadratic game~\eqref{eq:Ji-quadratic} with $N=4$ agents, $n_i=2$ decision variables each, and a vector parameter $p\in[-1,1]^2$ entering only the linear terms, $q_i(p)=c_{1i}p+c_{2i}$, while $A_{ii}$, $A_{i,-i}$ are constant. The blocks $A_{ii}$, $A_{i,-i}$, $c_{1i}$, $c_{2i}$ are drawn at random, and the diagonal blocks are shifted so that the symmetric part of the pseudogradient Jacobian $A$ has minimum eigenvalue exactly $\mu=0$, i.e., the game is monotone but not strongly monotone.

We first fit a monotone quadratic cost model of the form~\eqref{eq:quadratic_model} with $\mu=0$, each nonzero entry in $C,D$ constant with respect to $p$, and $q$ affine in $p$. We use $K=500$ cost samples $(x_k,p_k,\{\tilde J_i(x_k,p_k)\}_{i=1}^N)$, $k=1,\ldots,K$, with $x_k\in[-2,2]^8$, by minimizing $\LL_{DJ}(\theta)+\frac{\rho}{2}\|\theta\|^2$, $\theta\in\rr^{572}$. 
No NE-sample data is used ($K_{\rm NE}=0$) and
$K_{\rm val}=100$, $K_{\rm test}=200$ are generated similarly as held-out validation and test samples, respectively. The learned model is validated by computing its Nash equilibria at 50 different random test values of $p$ by solving the linear system $F(x,p)=0$, and comparing it against the true NE of the underlying game generating the data, obtained similarly 
for each test $p$.
The results are shown in the first row of Table~\ref{tab:quad_game}, which reports the mean BR and NE errors on test data.

Then, we address the inverse-learning problem described in Section~\ref{sec:inverse-quadratic}, generating the same number $K,K_{\rm val}$, and $K_{\rm test}$ of {\it best-response} data $(x_{k,-i_k},p_k,\bar x_{k,i_k})$. We compare three approaches: ($i$) the SDP~\eqref{eq:sdp-monotone} solved directly (``SDP''); ($ii$) the least-squares/SDP cascade~\eqref{eq:ls}--\eqref{eq:two-stage} (``LS+SDP''); ($iii$) the nonlinear least-squares problem~\eqref{eq:nls-monotone} 
with a small regularization $\rho=10^{-12}$ (``NLS''), all imposing $\mu=0$. 

The results are also summarized in Table~\ref{tab:quad_game}. The LS+SDP cascade is three orders of magnitude faster than solving the SDP directly, since it only requires an unconstrained linear least-squares fit followed by a small SDP. The NLS approach is the best in solving the inverse learning problem, but only marginally.

\begin{table}[t]
  \centering
  \setlength{\tabcolsep}{4.5pt}
  \renewcommand{\arraystretch}{1.} 
  \begin{tabular}{l|c|rrr}
  \hline
  method & data & time (s) & BR error & NE error \\
  \hline
  NLS & $J_{i,k}$ & 8.7263 & $2.62 \cdot 10^{-7}$ & $6.05 \cdot 10^{-7}$ \\
  SDP & $\bar x_{i,k}$ & 1.5016 & $2.79 \cdot 10^{-8}$ & $6.58 \cdot 10^{-8}$ \\
  LS+SDP & $\bar x_{i,k}$ & \textbf{0.0065} & $1.70 \cdot 10^{-8}$ & $3.30 \cdot 10^{-8}$ \\
  NLS & $\bar x_{i,k}$ & 8.3426 & $\mathbf{1.14 \cdot 10^{-8}}$ & $\mathbf{3.02 \cdot 10^{-8}}$ \\
  \hline
  \end{tabular}
  \caption{Comparison of learning approaches on randomly-generated parametric quadratic games.}
  \label{tab:quad_game}
\end{table}

\subsection{Monotone counterexample}
Consider the following parametric version of the monotone game of Remark~\ref{rmk:splitting_not_necessary}
\begin{equation}
\begin{aligned}
J_1(x,p)=&\tfrac12(x_1-p_1)^2+(x_1-p_1)\sin (x_2-p_2)\\
J_2(x,p)=&\tfrac12(x_2-p_2)^2-(x_2-p_2)\sin (x_1-p_1)
\end{aligned}
\label{eq:parametric-counterexample}
\end{equation}
that does not admit the exact parameterization of Section~\ref{sec:monotone_parametrization}. Its unconstrained NE is $x^\star(p)=p$. We learn the costs from $K=2000$ samples of $\DD_J$, drawn uniformly from $x\in[-\tfrac{\pi}{2},\tfrac{\pi}{2}]^2$ and $p\in[-\tfrac{\pi}{4},\tfrac{\pi}{4}]^2$, by minimizing~\eqref{eq:learning-problem-point} with $\LL_D=\LL_{DJ}$ as in~\eqref{eq:training_problem-J} with $\mu=0.2$. 

We compare four approaches: ($i$) monotonicity by construction, with costs parameterized as in~\eqref{eq:J_i-form}, $A(p)$ as in~\eqref{eq:pseudogradient-matrix}, $\Psi$ given by an input-convex neural network, and $C(p)$, $D(p)$, $q(p)$, $\phi_i$ by neural networks ($\LL_M=0$); ($ii$)--($iv$) input-convex neural-network costs $J_i$ with the monotonicity penalties $\LL_{M1}$~\eqref{eq:monotonicity-violation-1}, $\LL_{M2}$~\eqref{eq:monotonicity-violation-2}, and $\LL_{M3}$~\eqref{eq:monotonicity-violation-3}, respectively, with $\gamma=10^3$ and $M=2000$ unlabeled samples in $\DD_M$ ($M=50$ for $\LL_{M1}$, whose evaluation cost grows quadratically with $M$). All networks have two hidden layers of 4 neurons each and $K_{\rm val}=1000$ validation samples are used to select the best model. The learned NE is computed by solving $F(x,p;\theta)=0$ at 50 random test values of $p$ and compared with the exact NE $x^\star(p)=p$. Average fitting of the agents' costs is evaluated over $K_{\rm test}=2000$ test samples. 

As the penalty methods do not guarantee monotonicity, we also compute the minimum eigenvalue of the symmetric part of the learned pseudogradient Jacobian at the test samples and over the whole box by global optimization,
using the algorithm \texttt{DIRECT}~\cite{JSW98,JM21} for global optimization (absolute tolerance $10^{-8}$, relative tolerance $10^{-5}$, max 2000 evaluations). 

The obtained results are summarized in Table~\ref{tab:counterexample}. The best fit of the costs is obtained by the penalty methods $\LL_{M1}$ and $\LL_{M2}$, which also achieve the smallest NE error. However, they do not guarantee the imposed strong monotonicity $\mu=0.2$. The monotonicity-by-construction approach achieves a slightly worse fit of the costs and NE error, but satisfies the imposed monotonicity requirement. The penalty method $\LL_{M3}$ achieves a worse fit of the costs and NE error and a larger monotonicity constant. 
We report that lower values of $\mu$, such as $\mu=0$, may lead the penalty-based methods to learn non-monotone games, due to the limited density of samples in the training set and the value of the penalty $\gamma$.

\begin{table}[t]
  \centering
  \setlength{\tabcolsep}{5pt}
  \renewcommand{\arraystretch}{1.} 
   \begin{tabular}{l|rrrr}
  \hline
  method & time (s) & $\lambda_{\rm min}$ (test/global) & R$^2$ (\%) & NE error \\
  \hline
  monotonic~\eqref{eq:J_i-form} & \textbf{39.29} & $0.3087$ / $0.2741$ & 98.78 & $0.0424$ \\
  ICNN + $\LL_{M1}$~\eqref{eq:monotonicity-violation-1} & 50.93 & $0.0662$ / $0.0217$ & \textbf{99.99} & $\mathbf{0.0039}$ \\
  ICNN + $\LL_{M2}$~\eqref{eq:monotonicity-violation-2} & 108.38 & $0.1154$ / $0.0910$ & 99.98 & $0.0053$ \\
  ICNN + $\LL_{M3}$~\eqref{eq:monotonicity-violation-3} & 128.44 & $\mathbf{0.5105}$ / $\mathbf{0.4856}$ & 96.75 & $0.0528$ \\
  \hline
  \end{tabular}
  \caption{Comparison of learning approaches on example~\eqref{eq:parametric-counterexample}.}
      \label{tab:counterexample}
\end{table}

\subsection{Convex nonlinear game}
We consider the parametric variant of the internet switching model described in~\cite[Ex. A.1]{FK09}:
\begin{equation}
    \begin{aligned}
    J_i(x,q) = &-\frac{x_i}{x_1+...+x_N}(1-(x_1+...+x_N)q)\\
            &\sum_{i=1}^Nx_i\leq 1/q,\quad  x_i \geq 0.01
    \end{aligned}
\label{eq:A1-FK09}
\end{equation}
where $\ell=0.01$ and $q\in[0.5,1/(N\ell)]$ is treated as the parameter of the game.
This variant, with parameter $p=1/q$, was suggested in~\cite{BT26}, showing the convexity of costs $J_i$ and 
the analytical expression of the best response $\bar x_i(x_{-i},q)$. We generate random feasible
$K=K_{\rm val}=K_{\rm test}=4000$ training, validation, and test samples, respectively,
as described in~\cite{BT26}. We test both a pure quadratic monotone model~\eqref{eq:quadratic_model} with $C(p)$, $D(p)$, $q(p)$ given by (input-convex) neural networks, and a nonlinear monotone model given by adding an input-convex neural network potential term $\Psi$
with 2 hidden layers of 20 neurons each. No additional term $\phi_i$ is introduced and we set $\LL_M=0$. 

We evaluate and compare the best responses of the learned models against the exact best responses on a further 200 random test values of $p$, and solve both the GNEP~\eqref{eq:A1-FK09} and the GNEP with the learned models, under the same constraints, on 50 random test values of $p$. The results are summarized in Table~\ref{tab:nl_internet_game}. The learned models achieve a good fit of the costs, measured by the R$^2$-score on the test set averaged on all agents' costs, BR and NE errors.

\begin{table}[t]
  \centering
  \setlength{\tabcolsep}{4.5pt}
  \renewcommand{\arraystretch}{1.} 
    \begin{tabular}{c|rrrr}
  \hline
  potential & time (s) & R$^2$ (\%) & BR error & NE error \\
  \hline
  none & \textbf{72.5861} & 94.62 & $1.79 \times 10^{-2}$ & $\mathbf{3.60 \times 10^{-3}}$ \\
  NN & 300.6664 & \textbf{95.40} & $\mathbf{1.43 \times 10^{-2}}$ & $4.69 \times 10^{-3}$ \\
  \hline
  \end{tabular}
  \caption{Learned generalized Nash equilibrium models for the internet switching game with $N=5$ players.}
  \label{tab:nl_internet_game}
\end{table}

\subsection{Comparison with approximate explicit GNE method}
We compare the approach of Section~\ref{sec:monotone_parametrization} against the explicit-solution approach of~\cite{BT26}, which learns a neural-network approximation $\hat x(p)$ of the parametric map $p\mapsto x^\star(p)$ from best-response samples and knowledge of the true agents' costs $J_i$. 

\subsubsection{Parametric convex quadratic GNEP with shared linear constraints}
We first consider a randomly generated parametric monotone linear-quadratic GNEP (LQ-GNEP)~\eqref{eq:Ji-quadratic} with $N=3$ agents, $n_i=2$ decision variables each, $m=16$ shared inequality and $q=2$ shared equality constraints (both known), and strong-monotonicity constant $\mu=0.1$. Two parameters $p\in[-1,1]^2$ enter both the linear cost terms, $c_i(p)=c_{i0}+F_ip$, and the shared constraint right-hand side, $b(p)=b_0+Sp$ with $S>0$.

Using $K=2000$ best-response training samples, $K_{\rm val}=1000$ validation, and $K_{\rm test}=1000$ test samples generated as in~\cite{BT26} by projecting randomly-generated vectors $v_k\in[-10,10]^6$ on the feasible set, we train the explicit map $\hat x(p)$ of~\cite{BT26} using a NN with two hidden layers (30,20) neurons with swish activation per agent to approximate its associated value function, and a NN with (30,20) neurons and ReLU activation with linear bypass for the solution $\hat x(p)\in\rr^6$ (the latter has 888 weights to learn). We then fit a quadratic surrogate~\eqref{eq:quadratic_model}, with constant $A$ and $\mu=0$, on the same $(x_k,p_k)$ samples and the corresponding costs $\tilde J_i(x_k,p_k)$, by minimizing $\LL_{DJ}(\theta)+\frac{\rho}{2}\|\theta\|^2$. Note that only mere monotonicity is imposed on the surrogate, as the true $\mu$ is assumed unknown. At test time, the surrogate GNE is recovered online for each $p$ by solving the QP-GNEP with the same shared constraints $Ax\leq b(p)$, $Ex=h$ via the proximal-QP method of~\cite{BT26b} (using the \texttt{daqp} solver~\cite{ABA22b}).

Table~\ref{tab:surrogate_vs_explicit} compares the solutions obtained against the ground-truth GNE, computed by the same proximal-QP method, over the $1000$ test parameters. The surrogate is superior on every metric except online evaluation time (a forward pass through the solution NN $\hat x(p)$ vs. solving a QP to evaluate the GNE for each tested $p$). In particular, solving the surrogate GNE online always returns a feasible equilibrium, whereas the explicit map of~\cite{BT26}, being an unconstrained function approximator of a constrained map, slightly violates the shared constraints.
Note also that the approach of~\cite{BT26} requires fitting convex piecewise-quadratic value functions, while the approach of Section~\ref{sec:monotone_parametrization} has the easier task of fitting the agent's convex quadratic costs. On the other hand, the former approach is more general, as it can be applied to any parametric nonconvex GNEP, while the latter is limited to fitting parametric convex monotone GNEPs.

\begin{table}[t]
\centering
\setlength{\tabcolsep}{5pt}
\renewcommand{\arraystretch}{1.} 
\begin{tabular}{l|r|r}
\hline
 & explicit~\cite{BT26} & surrogate\\
\hline
best-response error (mean) & $ 3.4646 \cdot 10^{-1} $ & $ \mathbf{5.2079 \cdot 10^{-14}} $ \\
constraint violation (mean) & $ 3.4241 \cdot 10^{-4} $ & $ \mathbf{6.3303 \cdot 10^{-15}} $ \\
constraint violation (max) & $ 1.1252 \cdot 10^{-1} $ & $ \mathbf{2.5821 \cdot 10^{-13}} $ \\
opt.cost/cost fit (average R$^2$ score) & $ 0.9963 $ & $ \mathbf{1.0000} $ \\
fitting time (s) & $ 68.93 $ & $ \mathbf{11.92} $ \\
per-sample solve time ($\mu$s) & $ \mathbf{22.82} $ & $ 297.99 $ \\
\hline
\end{tabular}
\caption{Approximate explicit GNE solution~\cite{BT26} vs. LQ-GNE surrogates on a random strongly-monotone LQ-GNEP.}
\label{tab:surrogate_vs_explicit}
\end{table}

\subsubsection{Parametric nonlinear convex GNEP with shared linear constraints}
We repeat the comparison on a parametric nonlinear convex underlying game defined by
\begin{equation}
f_i(x,p)=\log\left(\sum_{r=1}^3e^{a_{ir}^\top x_i+b_{ir}^\top x_{-i}+\gamma_{ir}^\top p}\right)+\bigl(1+p_1+p_2\bigr)\sum_{i=1}^N\sum_{j=1}^{n_i}(\tfrac{1}{10}(x_i)_j)^4
\label{eq:nl-convex-GNE}
\end{equation}
with $N=2$ agents, $n_i=2$ decision variables each, $2$ parameters $p\in[.5,1]^2$, $a_{ir},b_{ir},\gamma_{ir}$ drawn at random, $m=10$ shared inequality constraints $A x\leq b(p)$, $b(p)=b_0+Sp$ with $S\geq 0$, and box constraints $-2\leq x\leq 2$. The game defined by~\eqref{eq:nl-convex-GNE} is not monotone for the selected values of vectors $b_{ir}$.

Both the training data for~\cite{BT26}'s explicit map and the ground-truth best responses used to score all methods are computed by a generic penalty-constrained L-BFGS-B solve of each agent's subproblem (via \texttt{nashopt.GNEP}~\cite{Bem25} 
We fit \cite{BT26}'s explicit map $\hat x(p)$ as in the previous section, using $K=2000$ training, $K_{\rm val}=1000$ validation, and $K_{\rm test}=1000$ test samples, and a monotone surrogate~\eqref{eq:quadratic_model} whose matrices $C(p)$, $D(p)$ and offset $q(p)$ are emitted by a hypernetwork (two hidden layers of $5$ neurons each, ReLU activation), augmented with an input-convex potential $\Psi(x,p)$ (two hidden layers of $5$ neurons each, softplus activation) common to all agents' costs. 

At test time, a v-GNE of the surrogate monotone game is recovered by running 
Korpelevich's extragradient method~\cite{Kor76} with step size $0.99/L(p)$, where the Lipschitz constant $L(p)$ of the pseudogradient of the learned surrogate game is computed as described in Proposition~\ref{prop:ICNN-lipschitz}. To properly initialize the method, we fit a second explicit map $\hat x_s(p)$ on value-function samples {\it of the surrogate game}, and use $\hat x_s(p)$ to initialize the extragradient method, which is then run for 50 iterations. 

Table~\ref{tab:surrogate_vs_explicit_nl} compares the two solutions on the test parameters. The surrogate attains a miuch more accurate best-response, near-zero constraint violation, and a higher cost-fit R$^2$ than the explicit map, at the cost of a more expensive online evaluation via Korpelevich's method rather than a single forward NN pass. 

\begin{table}[t]
\centering
\setlength{\tabcolsep}{5pt}
\renewcommand{\arraystretch}{1.} 
\begin{tabular}{l|r|r}
\hline
 & explicit~\cite{BT26} & surrogate\\
\hline
best-response error (mean) & $ 4.2598 \cdot 10^{-2} $ & $ \mathbf{1.2969 \cdot 10^{-5}} $ \\
constraint violation (mean) & $ 6.7551 \cdot 10^{-6} $ & $ \mathbf{9.3644 \cdot 10^{-9}} $ \\
constraint violation (max) & $ 3.5215 \cdot 10^{-3} $ & $ \mathbf{9.5764 \cdot 10^{-9}} $ \\
opt.cost/cost fit (average R$^2$ score) & $ 0.9834 $ & $ \mathbf{0.9970} $ \\
fitting time (s) & $ \mathbf{10.56} $ & $ 44.52 $ \\
per-sample solve time (s) & $ \mathbf{1.92 \cdot 10^{-5}} $ & $ 1.74 \cdot 10^{-1} $ \\
\hline
\end{tabular}
\caption{Approximate explicit GNE solution~\cite{BT26} vs. monotone surrogate (with potential term) on the nonlinear convex game~\eqref{eq:nl-convex-GNE}.}
\label{tab:surrogate_vs_explicit_nl}
\end{table}

\subsection{Inverse learning of a monotone game-theoretic controller}

\begin{figure}[t]
\begin{center}
\includegraphics[width=0.7\hsize]{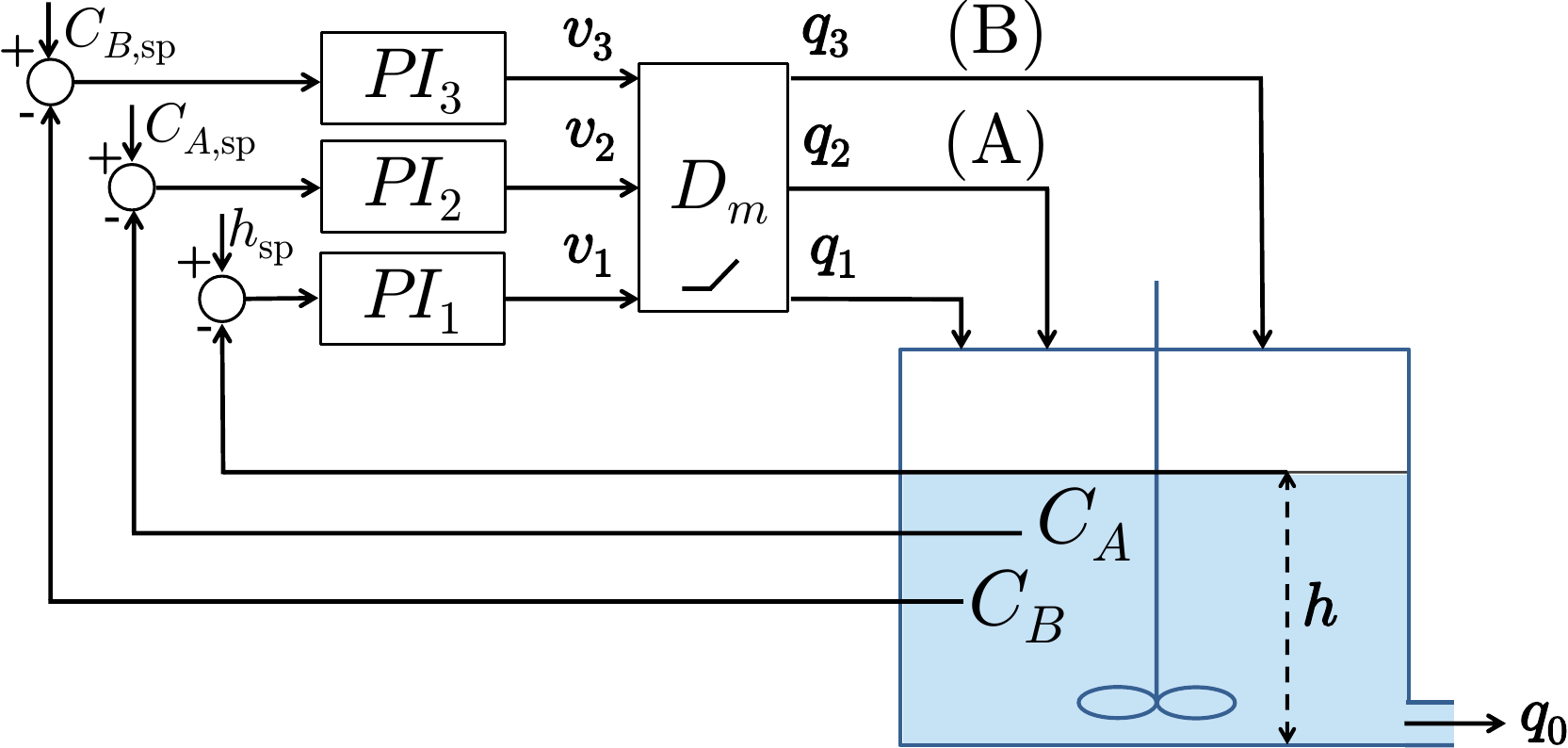}
\end{center}
\caption{Three-input stirred tank reactor.}
\label{fig:cstr}
\end{figure}

\begin{table}
\renewcommand{\arraystretch}{1.} 
\begin{center}
{\small
\begin{tabular}{@{}lll@{}}
 & Description & Nominal value \\
\hline
$h(t)$ & liquid level in the tank & $0.6$ [m] \\
$V(t)$ & liquid volume in the tank, $V = A h$ & $0.6$ [m$^3$] \\
$C_A(t)$ & concentration of reactant $A$ in the tank & $0.5$ [mol/L] \\
$C_B(t)$ & concentration of reactant $B$ in the tank & $0.3$ [mol/L] \\
$q_1(t)$ & flow rate of pure water (input) & $0.1086$ [m$^3$/s] \\
$q_2(t)$ & flow rate of the $A$-feed solution (input) & $0.1792$ [m$^3$/s] \\
$q_3(t)$ & flow rate of the $B$-feed solution (input) & $0.1770$ [m$^3$/s] \\
$q_{\mathrm{out}}(t)$ & outlet flow rate & $0.4648$ [m$^3$/s] \\
$C_{A,\mathrm{in}}$ & concentration of $A$ in stream 2 (parameter) & $2.0$ [mol/L] \\
$C_{B,\mathrm{in}}$ & concentration of $B$ in stream 3 (parameter) & $1.5$ [mol/L] \\
$S$ & tank cross-sectional area (parameter) & $1.0$ [m$^2$] \\
$c_v$ & Torricelli discharge coefficient (parameter) & $0.6$ [m$^{2.5}$/s] \\
$k$ & reaction rate constant (parameter) & $1.4$ [L/(mol\,s)] \\
\hline
\end{tabular}
}
\end{center}
\caption{Physical quantities and parameters of the stirred-tank reactor example and their nominal values.}
\label{tab:cstr-parameters}
\end{table}

We consider the three-input stirred tank depicted in Figure~\ref{fig:cstr}. The tank is fed by a pure-water stream $q_1$ and two reactant-feed streams $q_2$ (species $A$) and $q_3$ (species $B$) draining through 
an orifice, with a second-order reaction $A+B\to\text{product}$. 
The dynamics of the states $\xi=[h\ C_A\ C_B]^\top$ 
under the input $u=[q_1\ q_2\ q_3]^\top$ are described by the following system of nonlinear 
ordinary differential equations:
\begin{equation}
\left\{
\begin{aligned}
\dot h &= \frac{1}{S}(q_1 + q_2 + q_3 - c_v\sqrt{h}) \\
\dot C_A &= \frac{q_2(C_{A,\mathrm{in}}-C_A) - (q_1+q_3)C_A}{S h} - k C_A C_B \\
\dot C_B &= \frac{q_3(C_{B,\mathrm{in}}-C_B) - (q_1+q_2)C_B}{S h} - k C_A C_B 
\end{aligned}
\right.
\label{eq:cstr_dynamics}%
\end{equation}
where each physical quantity is described in Table~\ref{tab:cstr-parameters}. We control the system by three decentralized single-input single-output PI loops regulating the three outputs $h$, $C_A$, and $C_B$ to the corresponding setpoints $h_{\rm sp}$, $C_{A,\rm sp}$, and $C_{B,\rm sp}$ by manipulating the three flows $q_1$, $q_2$, and $q_3$,
respectively. To this end, we linearize~\eqref{eq:cstr_dynamics} and get $\dot\xi = A_c (\xi-\xi_0) + B_c (u-u_0)$, where $\xi_0$, $u_0$ are the nominal operating points also reported in Table~\ref{tab:cstr-parameters}. To have each PI loop treat the other inputs as measured disturbances, we decouple the three loops by a static feedforward matrix $D_m = B_c^{-1}\mathrm{diag}(B_c)$. The resulting PI control laws are:
\begin{subequations}
\begin{align}
e(t) &= \begin{bmatrix} h_{\rm sp}(t)-h(t)\\ C_{A,\rm sp}(t)-C_A(t)\\ C_{B,\rm sp}(t)-C_B(t)\end{bmatrix} \qquad
\begin{aligned}
v(t) &= K_p\,e(t) + K_i\,I(t)\\[.5em]
\dot I_i(t)&=e_i(t) \label{eq:pi-law}
\end{aligned}\\
u(t) &= \max\big(u_0 + D_m\,v(t), 0\big) \label{eq:pi-decoupling}
\end{align}%
\label{eq:pi-controller}%
\end{subequations}
with $K_p=\mathrm{diag}(1.5492,\,0.0956,\,0.0737)$, $K_i=\mathrm{diag}(0.6000,\,0.1142,\,0.1087)$.

We generate the training, validation, and test data by simulating the decentralized PI closed loop~\eqref{eq:cstr_dynamics}--\eqref{eq:pi-controller} for $T_{\rm sim}=1999$\,s under piecewise-constant random setpoints $h_{\rm sp},C_{A,\rm sp},C_{B,\rm sp}$ that change every $50$\,s to a new value drawn uniformly around $h_0,C_{A0},C_{B0}$, sampling the state, setpoints, and PI integrator states every $T_s=1$\,s. Our working assumption is to consider each of the resulting $T_{\rm sim}/T_s+1=2000$ input samples as a best-response $\bar x_k=[q_1\ q_2\ q_3](t_k)$ to the parameter $p_k=[h,C_A,C_B,h_{\rm sp},C_{A,\rm sp},C_{B,\rm sp},I_h,I_A,I_B](t_k)$, per agent $i\in\{1,2,3\}$, for a total of $3\times2000=6000$ samples, which we split randomly into $80\%$ training, $10\%$ validation, and $10\%$ test sets ($4800$, $600$, and $600$ rows, respectively).

Next, we learn a monotone quadratic game-theoretic model~\eqref{eq:J_i-form} of the three agents' best responses by solving the inverse problem~\eqref{eq:sdp-monotone} with $\mu=0$, $\LL_M=0$, using the least-squares/SDP cascade~\eqref{eq:ls}--\eqref{eq:two-stage}, obtaining an average best-response error of $1.55 \cdot 10^{-4}$. The corresponding symmetric part of the learned constant pseudogradient matrix $A$ is positive semidefinite with eigenvalues $0.6135$, $1.1273$, and $1.2592$. Finally, we test the closed-loop system obtained by replacing the three PI loops with a game-theoretic controller that computes the GNE of the learned monotone game at each time step under the shared constraints $0.4\leq q_1+q_2+q_3\leq 0.5$  and local constraints $q_i\geq 0$, $i=1,2,3$ using the proximal-QP method proposed in~\cite{BT26b} and the DAQP solver~\cite{ABA22b} to solve each QP instance. The results are shown in Figure~\ref{fig:cstr-control}. The learned game-theoretic controller achieves a performance very close to the original PI controller within the admissible range of total flow constraints,
and is able to enforce the shared constraint on the total flow in a game-theoretic manner, which was not possible with the decentralized PI controller.

\begin{figure}
    \centering
    \begin{subfigure}{.7\hsize}
        \centering
        \includegraphics[width=\hsize]{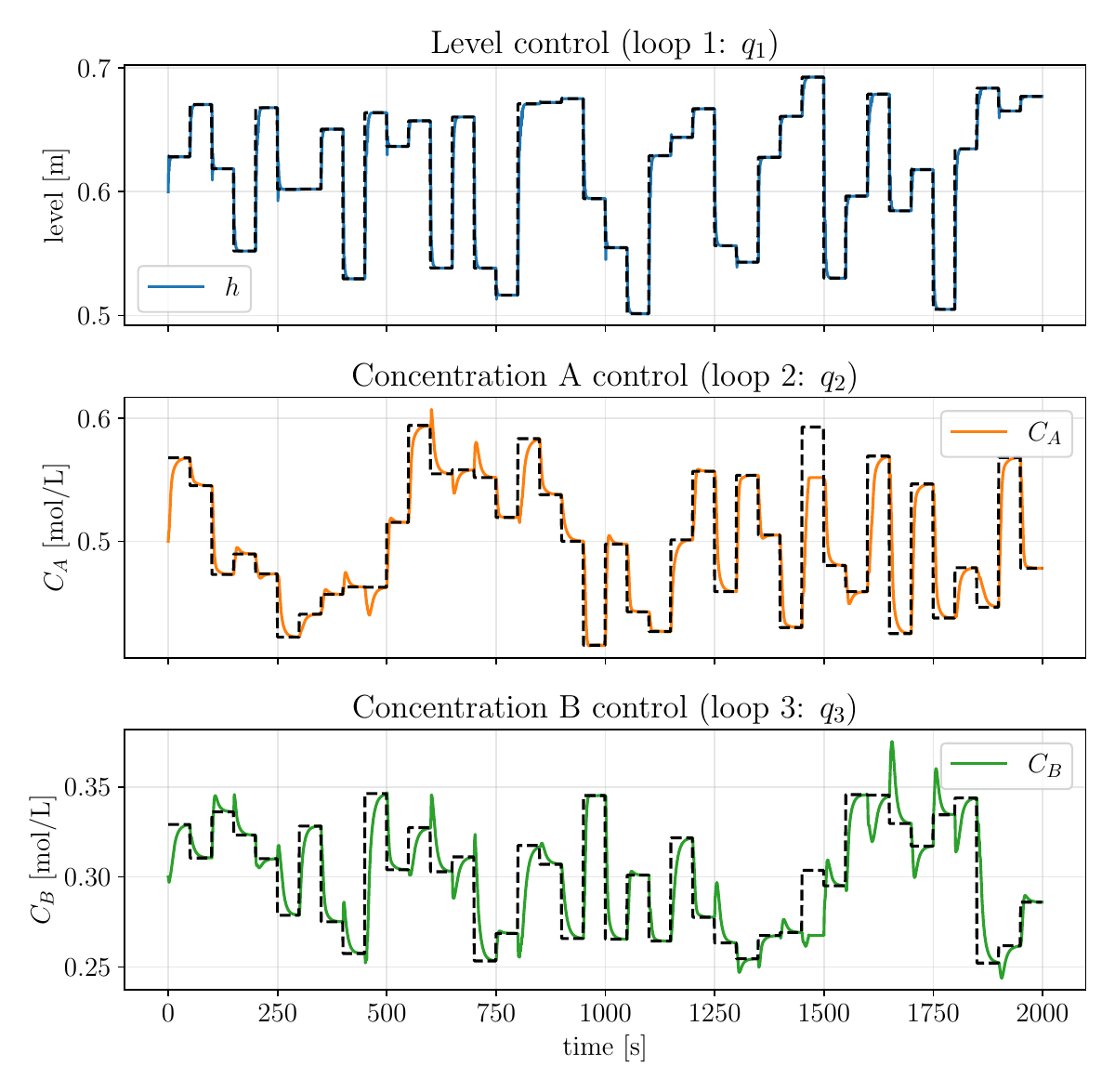}
        \caption{Controlled outputs.}
        \label{fig:cstr-control-outputs}
    \end{subfigure}
    \begin{subfigure}{.7\hsize}
        \centering
        \includegraphics[width=\hsize]{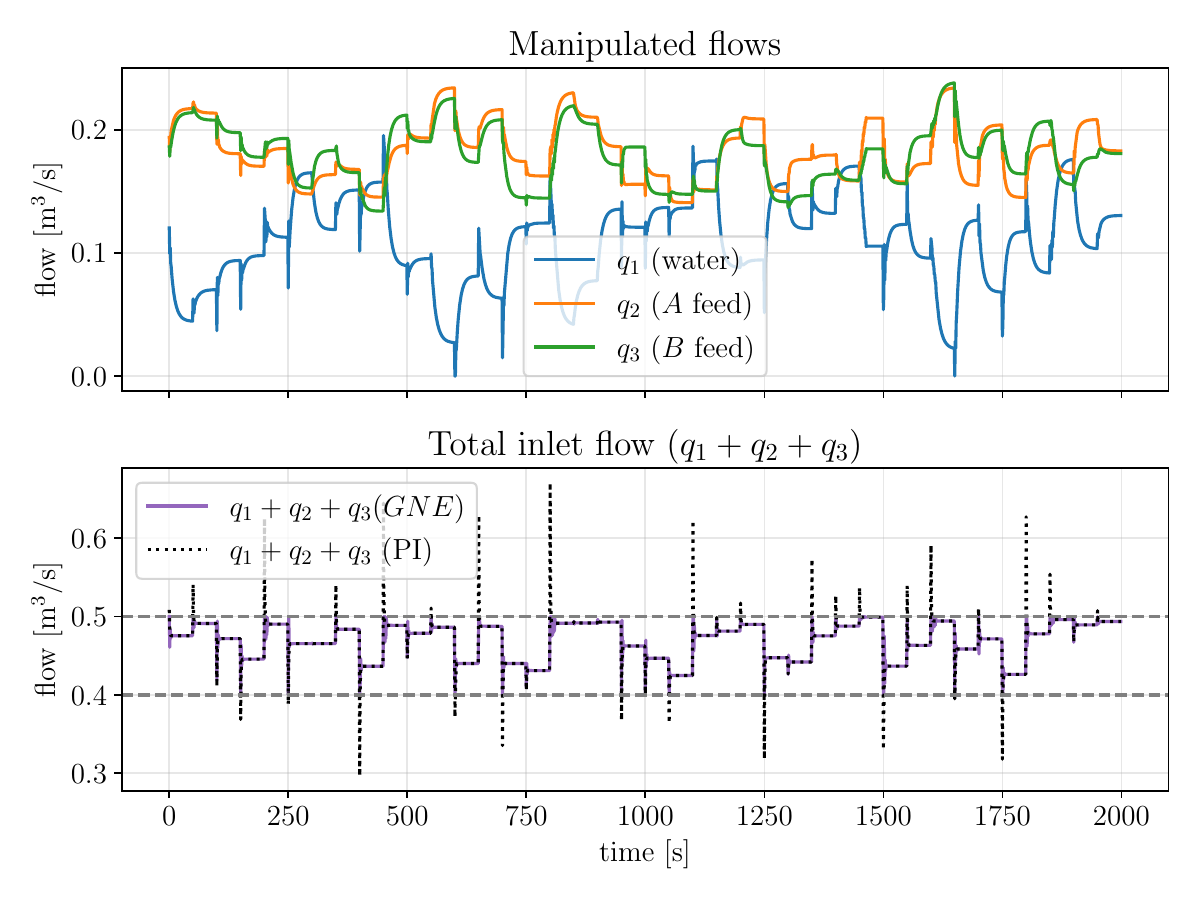}
        \caption{Command inputs (GNE vs PI control).}
        \label{fig:cstr-control-inputs}
    \end{subfigure}
    \caption{Closed-loop response of the stirred-tank reactor under the discrete-time GNE-based game-theoretic controller.}
    \label{fig:cstr-control}
\end{figure}

\section{Conclusions}
\label{sec:conclusions}
We addressed the problem of learning a parametric NE problem, directly from agents' cost values or from their best responses, satisfying given monotonicity requirements for all parameter values. A penalty-based method promotes monotonicity on training data without certifying it elsewhere, while a structured parameterization, based on a representation of the pseudogradient as a convex potential plus a skew-symmetric affine term, guarantees monotonicity by construction and covers all monotone convex quadratic games, whose inverse learning reduces to an SDP. Numerical results confirm that, in spite of the limitations of the chosen model classes, convex monotone parametric games can be learned rather accurately.

Future work includes extending the representation to other classes of learnable monotone games that include state-dependent skew-symmetric terms, and to applications of the proposed methods to learn monotone game-theoretic model predictive control laws from observed input/state trajectories.


\begin{thebibliography}{10}

\bibitem{AABBDK19}
A.~Agrawal, B.~Amos, S.~Barratt, S.~Boyd, S.~Diamond, and J.Z. Kolter.
\newblock Differentiable convex optimization layers.
\newblock {\em Advances in neural information processing systems}, 32, 2019.

\bibitem{ABA22b}
D.~Arnstr\"{o}m, A.~Bemporad, and D.~Axehill.
\newblock A dual active-set solver for embedded quadratic programming using recursive {LDL$^{\rm T}$} updates.
\newblock 67(8):4362--4369, 2022.

\bibitem{Bau16}
D.~Bauso.
\newblock {\em Game Theory with Engineering Applications}.
\newblock SIAM, 2016.

\bibitem{belgioioso2020}
G.~Belgioioso, A.~Ned\'i\v{c}, and S.~Grammatico.
\newblock Distributed generalized {Nash} equilibrium seeking in aggregative games on time-varying networks.
\newblock {\em IEEE Transactions on Automatic Control}, 66(6):2582--2597, 2021.

\bibitem{BYGP22}
G.~Belgioioso, P.~Yi, S.~Grammatico, and L.~Pavel.
\newblock Distributed generalized {N}ash equilibrium seeking: An operator-theoretic perspective.
\newblock {\em IEEE Control Systems Magazine}, 42(4):87--102, 2022.

\bibitem{Bem25b}
A.~Bemporad.
\newblock An {L-BFGS-B} approach for linear and nonlinear system identification under $\ell_1$ and group-lasso regularization.
\newblock {\em IEEE Transactions on Automatic Control}, 70(7):4857--4864, 2025.

\bibitem{Bem25}
A.~Bemporad.
\newblock {NashOpt}: A {Python} library for computing generalized {Nash} equilibria and game design.
\newblock 2025.
\newblock arXiv preprint 2512.23636. Code available at \url{https://github.com/bemporad/nashopt}.

\bibitem{BT26}
A.~Bemporad and T.~Tatarenko.
\newblock Learning approximate solutions to multiparametric generalized {Nash} equilibrium problems.
\newblock 2026.
\newblock available on arXiv at \url{https://arxiv.org/abs/2605.28757}. Code available at \url{https://github.com/bemporad/mpfit}.

\bibitem{BT26b}
A.~Bemporad and T.~Tatarenko.
\newblock Solving monotone linear-quadratic generalized {Nash} equilibrium problems via quadratic programming.
\newblock 2026.
\newblock available on arXiv at \url{https://arxiv.org/abs/2608.07336}.

\bibitem{BV04}
S.~Boyd and L.~Vandenberghe.
\newblock {\em Convex Optimization}.
\newblock Cambridge University Press, New York, NY, USA, 2004.

\bibitem{LJWA20}
H.~Le Cadre, P.~Jacquot, C.~Wan, and C.~Alasseur.
\newblock Peer-to-peer electricity market analysis: From variational to generalized {N}ash equilibrium.
\newblock {\em European Journal of Operational Research}, 282:753--771, 2020.

\bibitem{Cappello_TCNS_2021}
D.~Cappello and T.~Mylvaganam.
\newblock Distributed differential games for control of multi-agent systems.
\newblock {\em IEEE Transactions on Control of Network Systems}, 9(2):635--646, 2022.

\bibitem{DBS05}
M.~Diehl, H.G. Bock, and J.P. Schl{\"o}der.
\newblock A real-time iteration scheme for nonlinear optimization in optimal feedback control.
\newblock {\em SIAM Journal on Control and Optimization}, 43(5):1714--1736, 2005.

\bibitem{Dockner_GameEco_2000}
E.~Dockner, S.~J{\o}rgensen, N.~Van Long, and G.~Sorger.
\newblock {\em Differential games in economics and management science}.
\newblock Cambridge University Press, Cambridge, U.K., 2000.

\bibitem{ESF24}
T.~Entesari, S.~Sharifi, and M.~Fazlyab.
\newblock Compositional curvature bounds for deep neural networks.
\newblock In {\em International conference on machine learning}, pages 12527--12546. PMLR, 2024.

\bibitem{FK09}
F.~Facchinei and C.~Kanzow.
\newblock Penalty methods for the solution of generalized {Nash} equilibrium problems (with complete test problems).
\newblock Technical report, Institute of Mathematics, University of W{\"u}rzburg, 2009.

\bibitem{facchinei2010generalized}
F.~Facchinei and C.~Kanzow.
\newblock Generalized {N}ash equilibrium problems.
\newblock {\em Annals of Operations Research}, 175(1):177--211, 2010.

\bibitem{FP03}
F.~Facchinei and J.-S. Pang.
\newblock {\em Finite-dimensional variational inequalities and complementarity problems}.
\newblock Springer, 2003.

\bibitem{goktas2023generative}
D.~Goktas, D.C. Parkes, I.~Gemp, L.~Marris, G.~Piliouras, R.~Elie, G.~Lever, and A.~Tacchetti.
\newblock Generative adversarial equilibrium solvers.
\newblock {\em arXiv preprint arXiv:2302.06607}, 2023.

\bibitem{Hall_CDC_2022}
S.~Hall, G.~Belgioioso, D.~Liao-McPherson, and F.~Dorfler.
\newblock Receding horizon games with coupling constraints for demand-side management.
\newblock In {\em IEEE 61st Conference on Decision and Control (CDC)}, pages 3795--3800, 2022.

\bibitem{han2024noregret}
M.~Han, F.~Zhang, and Y.~Chen.
\newblock No-regret learning of {N}ash equilibrium for black-box games via {G}aussian processes.
\newblock In {\em Proceedings of the Fortieth Conference on Uncertainty in Artificial Intelligence}, pages 1541--1557, 2024.
\newblock arXiv:2405.08318.

\bibitem{JM21}
D.R. Jones and J.R.R.A. Martins.
\newblock The {DIRECT} algorithm: 25 years later.
\newblock {\em Journal of Global Optimization}, 79(3):521--566, 2021.

\bibitem{JSW98}
D.R. Jones, M.~Schonlau, and W.J. Matthias.
\newblock Efficient global optimization of expensive black-box functions.
\newblock {\em Journal of Global Optimization}, 13(4):455--492, 1998.

\bibitem{Kor76}
G.M. Korpelevich.
\newblock The extragradient method for finding saddle points and other problems.
\newblock {\em Matecon}, 12:747--756, 1976.

\bibitem{KLL22}
G.~Kotsalis, G.~Lan, and T.~Li.
\newblock Simple and optimal methods for stochastic variational inequalities, {I}: Operator extrapolation.
\newblock {\em SIAM Journal on Optimization}, 32(3):2041--2073, 2022.

\bibitem{krupa2026learning}
P.~Krupa and A.~Bemporad.
\newblock Learning generalized {Nash} equilibria from pairwise preferences.
\newblock {\em IEEE Control Systems Letters and CDC'26}, 10:1045--1050.

\bibitem{Lju99}
L.~Ljung.
\newblock {\em System Identification : Theory for the User}.
\newblock Prentice Hall, 2 edition, 1999.

\bibitem{MSA16}
Y.~Mao, M.~Szmuk, and B.~A{\c{c}}{\i}kme{\c{s}}e.
\newblock Successive convexification of non-convex optimal control problems and its convergence properties.
\newblock In {\em IEEE 55th Conference on Decision and Control}, pages 3636--3641, 2016.

\bibitem{NS11}
Y.~Nesterov and L.~Scrimali.
\newblock Solving strongly monotone variational and quasi-variational inequalities.
\newblock {\em Discrete and Continuous Dynamical Systems}, 31(4):1383--1396, 2011.

\bibitem{SBB25}
M.~Schaller, A.~Bemporad, and S.~Boyd.
\newblock Learning parametric convex functions.
\newblock 2025.
\newblock \url{http://arxiv.org/abs/2506.04183}.

\bibitem{SF20}
S.~Singla and S.~Feizi.
\newblock Second-order provable defenses against adversarial attacks.
\newblock In {\em International conference on machine learning}, pages 8981--8991. PMLR, 2020.

\bibitem{TatKam25}
T.~Tatarenko and M.~Kamgarpour.
\newblock Convergence rate of payoff-based generalized nash equilibrium learning.
\newblock {\em European Journal of Control}, 86:101372, 2025.
\newblock Special Issue on the European Control Conference 2025.

\bibitem{TatNed25}
T.~Tatarenko and A.~Nedich.
\newblock Fast distributed {Nash} equilibrium seeking in monotone games.
\newblock {\em arXiv, 2507.11703}, 2025.

\bibitem{TM24}
K.~Tracy and Z.~Manchester.
\newblock On the differentiability of the primal-dual interior-point method.
\newblock {\em arXiv preprint 2406.11749}, 2024.
\newblock \url{https://github.com/qpax-solver/qpax}.

\bibitem{WLMCMWW21}
Z.~Wang, F.~Liu, Z.~Ma, Y.~Chen, M.~Jia, W.~Wei, and Q.~Wu.
\newblock Distributed generalized {N}ash equilibrium seeking for energy sharing games in prosumers.
\newblock {\em IEEE Transactions on Power Systems}, 36(5):3973--3986, 2021.

\end{thebibliography}
\end{document}